\documentclass[a4work,10pt,reqno]{amsart}

\usepackage{xpatch}
\xpatchcmd{\paragraph}{\normalfont}{{\normalfont\itshape}}{}{}
\usepackage{hyperref}
\usepackage{amssymb}
\usepackage{eucal}
\usepackage{eufrak}
\usepackage[labelfont=bf]{caption}
\usepackage{a4wide}
\usepackage{graphics}
\usepackage{graphicx}
\usepackage{multicol}
\usepackage{amsmath,mathtools}
\usepackage{wasysym}
\usepackage{dsfont}
\usepackage{float}
\usepackage{fancyhdr}
\usepackage{subfigure}
\usepackage{comment}
\usepackage{tikz}
 \usepackage{algorithm}
 \usepackage{algpseudocode}

\usepackage[T1]{fontenc}
\usepackage[utf8]{inputenc}

\usepackage{graphics}
\usepackage{graphicx}
\usepackage{epsfig}
\usepackage{multicol}
\usepackage{amsmath}
\usepackage{mathrsfs}
\usepackage{color}

\newtheorem{theorem}{Theorem}[section]
\newtheorem{lemma}{Lemma}[section]

\theoremstyle{definition}
\newtheorem{definition}{Definition}[section]
\newtheorem{remark}{Remark}[section]

\graphicspath{ {./Pictures/} }

\newcommand{\lp}{\left(}
\newcommand{\rp}{\right)}
\newcommand{\lc}{\left\{}
\newcommand{\rc}{\right\}}
\newcommand{\der}{\partial}

\newcommand{\R}{\mathbb{R}}      

\newcommand{\N}{\mathbb{N}}      
\newcommand{\C}{\mathbb{C}}      

\newcommand{\Flder}{\rightarrow}

\usepackage{amssymb}

\newcommand{\proa}{A^*G \mbox{$\;$}_{\tau^*} \kern-3pt\times_\alpha
G \mbox{$\;$}_\beta \kern-3pt\times_{\tau^*} A^*G}

\newcommand{\Sc}{\mathcal{S}^{^\text{cons}}}
\newcommand{\Sf}{\mathcal{S}^{^\text{frac}}}
\newcommand{\Sg}{\mathcal{S}}

\usepackage[
    backend=biber,  style=ieee, sorting=nyt
  ]{biblatex}

\xpatchbibmacro{volume+number+eid}{%
  \setunit*{\adddot}%
}{%
}{}{}

\DeclareFieldFormat[article,book]{volume}{\mkbibbold{#1}}
\DeclareFieldFormat[article]{number}{\mkbibbold{\mkbibparens{#1}}}
\DeclareFieldFormat[book]{number}{\mkbibbold{#1}}

    \DeclareFieldFormat{parens}{\mkbibbold{\mkbibparens{#1}}}
    \DeclareFieldFormat{date}{\mkbibbold{\mkbibparens{#1}}}
    
\renewbibmacro*{volume+number+eid}{%
  \printfield{volume}%
  \setunit*{}
  \printfield{number}%
  \setunit{\addcomma}%
  \printfield{eid}}

\DeclareFieldFormat[book]{title}{\textit{#1}}

\DeclareFieldFormat[book,inbook,incollection,collection]{series}{\normalfont{#1}}

\renewbibmacro*{journal+issuetitle}{%
  \usebibmacro{journal}%
  \setunit*{\addspace}%
  \iffieldundef{series}
    {}
    {\newunit
     \printfield{series}%
     \setunit{\addspace}}%
  \usebibmacro{volume+number+eid}%
  \setunit{\addspace}%
  \usebibmacro{issue}%
  \newunit}

  \newbibmacro*{series+number(book)}{%
  \iffieldundef{series}{}
    {%
      \setunit{\addcomma\addspace}%
      \printtext[]{%
        \printfield{series}%
        \setunit*{\addspace}%
        \printfield{number}%
        \newunit}}}

\begin{document}

\title{Fractional variational integrators based on convolution quadrature}

\begin{abstract}
Fractional dissipation is a powerful tool to study non-local physical phenomena such as damping models. The design of geometric, in particular, variational integrators for the numerical simulation of such systems relies on a variational formulation of the model. In \cite{JiOb2}, a new approach is proposed to deal with dissipative systems including fractionally damped systems in a variational way for both, the continuous and discrete setting. It is based on the doubling of variables and their fractional derivatives. The aim of this work is to derive higher-order fractional variational integrators by means of convolution quadrature (CQ) based on backward difference formulas.  We then provide numerical methods that are of order 2 improving a previous result in \cite{JiOb2}. The convergence properties of the fractional variational integrators and saturation effects due to the approximation of the fractional derivatives by CQ are studied numerically.
\end{abstract}



\author{Khaled Hariz Belgacem}

\email{hariz@math.upb.de}

\address{Department of Mathematics, University of Paderborn, Warburger Straße 100, 33098 Paderborn,
Germany}

\author{Fernando Jim\'enez}
\address{Departamento de Matemática Aplicada I, ETSII, Universidad Nacional de Educación a Distancia (UNED)
c. Juan del Rosal 12, 28040, Madrid, Spain.}

\email{fjimenez@ind.uned.es}

\author{Sina Ober-Bl\"obaum}
\email{sinaober@math.upb.de}
\address{Department of Mathematics, University of Paderborn, Warburger Straße 100, 33098 Paderborn,
Germany}

\maketitle

\section{Introduction}
Fractional calculus is an extension of classical integration and differentiation theory to any real or complex order \cite{TheBook,TheBook2,Ignor,Oldham}. A major feature enjoyed by fractional calculus is non-locality which is used widely  to model numerous  phenomena in mechanics and physics. Fractional damping systems including dissipation can be considered as a variational problem  with a  Lagrangian which depends on fractional derivatives, so that the corresponding equations of motion arise from  Hamilton's principle. In this context, a new variational approach, the so-called {\em restricted Hamilton's principle} was developed by Jim\'enez and Ober-Bl\"obaum \cite{JiOb2}. Contrary to the classical Hamilton's principle,  the restricted one gives only {\em sufficient} conditions for the extremals of the fractional variational problem leading to {\em  restricted fractional Euler-Lagrange equation:}
\begin{equation}\label{eq:FV}
    \frac{d}{dt}\lp\frac{\der L}{\der\dot x}\rp-\frac{\der L}{\der x}=-\mu\,D^{\alpha}_{-}x,\quad \mu >0,\quad \alpha \geq 0
\end{equation}
where $L(t,x,\dot x)$ is a Lagrangian and $D_-^{\alpha}$ is the fractional derivative operator.\\  

There are several interesting applications of the fractional dynamical equation \eqref{eq:FV}. For example, it  can be used to describe the dynamics of damped linear system when $\alpha=1$, i.e.~ the fractional derivative becomes the full derivative or the motion of rigid plate immersed in a Newtonian fluid for fractional  derivatives of order $3/2$ \cite{Torvik}.\\

Dealing with variational problems permit us to construct variational integrators \cite{MaWe,LuGeWa} via a discrete calculus of variations which are  numerical schemes for Lagrangian systems preserving their variational
structures. For this purpose, it is important to derive of the equations of motion for forced systems (equation \eqref{eq:FV}  with $\alpha=1$) in a purely variational way. There exist several attempts in this direction, some of those based on duplicating the variables of the system \cite{Galley,Diego}.
The construction of the desired forced variational  integrators was motivated by \cite{Diego}.
. \\

The notion of variational integrators to the fractional case has been discussed \cite{JiOb1} using a discrete restricted Hamilton's principle. This approach  based on the Grünwald-Letnikov approximation of the fractional  derivatives, is also developed in \cite{JiOb2}.  Such approximation  has been proved to be of order one consistency, so that   the convergence order the resulting scheme, called {\em Fractional Variational Integrator} (FVI)  is then limited by  $1$. Following the previous work  \cite{JiOb2}, our purpose is to derive high-order  variational integrators for \eqref{eq:FV}  by combining  high-order variational techniques  \cite{SinaSaake,O14}  with {\em convolution quadrature} (CQ) \cite{Lubich1,Lubich2}.\\

Among several  numerical methods used in the fractional framework,  convolution quadrature preserves structure. More concretely,  there are two  important properties of fractional operators used in the restricted Hamilton's principle: integration by parts and semigroup properties. For the construction of FVI, preserving such properties at the discrete setting  is then essential which can be done by CQ.\\

Convolution quadrature was introduced by  Lubich \cite{Lubich1,Lubich2}. This method  is a numerical tool for  approximating convolution integral  by a specific quadrature rule.  The main difference between this method and  other numerical methods is that the  weights of CQ are computed by Laplace transform of the convolution kernel and  multistep methods. For  the {\em left} fractional integral, which is a particular convolution integral,  the quadrature weights are obtained from the fractional order power of the rational polynomial of the generating functions of LMMs.  In particular,  the use of backward difference  formulas (BDFs) is a subclass of LMMs which is widely adopted for high accuracy \cite{LubichF1,LubichF2}.\\

The  FVIs is mainly a combination of two different algorithms: one for fractional part and the other for conservative part.  As we deal with higher order approximation, the above listed strategy is appropriate for the fractional derivative involved in \eqref{eq:FV}. Besides that, it is also natural to apply  higher order approximation for the conservative part and this  motivates us to  use high-order variational techniques.\\

High-order variational techniques, also known as high-order variational integrators or Galerkin variational methods are
 numerical approaches applied to the action integral associated to a Lagrangian $L$ in order to construct numerical schemes of arbitrarily high order. It is based on interpolating the trajectories and choosing  a high-order quadrature for the approximation of the integral, see  \cite{MaWe,Leok2011,SinaSaake} and references therein.\\

The work is organized as follows. In Section \ref{sec:higher-orderDVM} we  give a brief exposition of Lagrangian variational integrators that  will be used throughout the work. Section \ref{FracIntegrals} contains some necessary preliminaries of fractional calculus,  and we present  the notion of convolution quadrature, in particular, Lubich's fractional linear multi-step methods. We also discuss the main issue of using convolution quadrature together based on BDFs  with numerical experiments as illustration.  After recalling a continuous restricted Hamilton's principle in Section \ref{sec:RVP}, we present our new contribution as a generalisation of the one obtained by \cite{JiOb2} for deriving FVIs associated to \eqref{eq:FV}  by means of the convolution quadrature in the framework of the discrete restricted Hamilton's principle.  We examine the accuracy of such integrators using the damped oscillator and the Bagley-Torvik problems. The final Section we treat the case higher order FVIs by applying  high-order techniques as presented in \cite{SinaSaake,O14} mixed with convolution quadrature to obtain higher order fractional variational integrators.


\section{Higher order discrete variational mechanics}\label{sec:higher-orderDVM}

In this section, we remind the construction of variational integrators  which will be used in this work. We refer
to \cite{MaWe,Leok2011,SinaSaake} and references therein for more details.

\subsection{Hamilton's principle and Euler-Lagrange equations}

Consider a mechanical system defined on the $d$-dimensional configuration manifold $Q$ (later on we will particularize on $\R^d$, but in this section it can be considered as a general smooth manifold) with corresponding tangent bundle $TQ$. Let $q(t)\in Q$ and $\dot q(t)\in T_{q(t)}Q$, $t\in [0,T]\subset \R$, $0<T$,  be the time-dependent configuration and velocity of the system. The action $S:C^2([0,T],Q)\Flder\R$ of a mechanical system is defined as the time integral of the Lagrangian, i.e.~
\begin{equation}\label{ContAc}
\mathcal{L}(q)=\int_0^TL(q(t),\dot q(t))\,dt,
\end{equation}
where the $C^2$ Lagrangian function $L:TQ\Flder\R$ consists of kinetic minus potential energy. Hamilton's principle seeks curves $q$, with fixed initial and final values $q(0)$ and $q(T)$, which are {\it extremals} of the action,  i.e.~ satisfying
\[
\delta\, \mathcal{L}(q)=0,
\]
for arbitrary variations $\delta q\in T_qC^2([0,T],Q).$ A necessary and sufficient condition for the extremals is the so-called Euler-Lagrange equation 
\begin{equation}\label{EL}
\frac{d}{dt}\lp\frac{\der L}{\der\dot q}\rp-\frac{\der L}{\der q}=0,
\end{equation}
which are a second-order differential equation and describes the dynamics of conservative systems. See \cite{AbMa} for more details.

\subsection{Discrete Hamilton's principle and discrete Euler-Lagrange equations}\label{Disc}
The discretization of the objects described in the previous subsection is based on \cite{MaWe,MoVe}. Let us consider a time grid $t_k=\lc k\,h\,| k=0,\ldots, N\rc$, where $h\in\R_+$ is the time step and $h\,N=T$. We replace the configuration $q(t)$ by a discrete sequence $q_d\equiv\lc q_k\rc_{0:N}\in Q^{N+1}$ where $q_k$ will be an approximation of $q(t_k)$. The discrete Lagrangian $L_d:Q\times Q\Flder\R$ will be an approximation of the action \eqref{ContAc} in one time step $[t_k,t_{k+1}]$ based on two neighboring configurations $q_k$ and $q_{k+1}$, i.e.~
\begin{equation}\label{TwoPointsApprox}
L_d(q_k,q_{k+1})\approx \int_{t_k}^{t_{k+1}}L(q(t),\dot q(t))\,dt.
\end{equation}
Furthermore, the discrete action sum $\mathcal{L}_d:Q^{N+1}\Flder\R$ is defined by
\[
\mathcal{S}_d(q_d)=\sum_{k=0}^{N-1}L_d(q_k,q_{k+1}).
\]
The discrete Hamilton's principle seeks extremals of the action $\mathcal{L}_d$ with fixed endpoints $q_0,\,q_N$, i.e.~
\[
\delta \mathcal{S}_d(q_d)=0,
\]
for arbitrary $\delta q_k\in T_{q_k}Q$. A necessary and sufficient condition for the extremals are the discrete Euler-Lagrange equations
\begin{equation}\label{DEL}
D_1L_d(q_k,q_{k+1})+D_2L_d(q_{k-1},q_{k})=0,\quad\quad k=1,\ldots, N-1,
\end{equation}
where $D_i$ is the derivative with respect to the $i$-th argument. Given that the matrix $D_{12}L_d(q_k,q_{k+1})$ is regular, equation \eqref{DEL} provides a discrete iteration scheme for \eqref{EL} that determines $q_{k+1}$ for given $q_{k}$ and $q_{k-1}$. This iteration scheme, that can be represented by the map  $F_{L_d}:Q\times Q\Flder Q\times Q$, $(q_{k-1},q_k)\mapsto (q_k,q_{k+1})$, is called {\it variational integrator}, and has interesting preservation properties, such as symplecticity and momentum preservation under the action of a symmetry \cite{MaWe,MoVe}.

\subsection{Higher order approximations of the action}\label{HO-Action}

Considering only two neighboring configurations $q_k$ and $q_{k+1}$ in \eqref{TwoPointsApprox} limits the  approximation order to $O(h^2)$. With the aim of increasing this order, a well-known approach is to take into account inner nodes in between $[t_k, t_{k+1}]$ \cite{MaWe,SinaSaake,HaLe13}. This higher order approximation procedure consists of two steps:  (1) the approximation of the space of trajectories and (2) the approximation of the integral of the Lagrangian by appropriate quadrature rules.
\medskip

\paragraph{(1) Trajectories  space} The  space 
$\mathcal{C}([t_k,t_{k+1}],Q)=\lc q:[t_k,t_{k+1}]\Flder Q\,|\, q(t_k)=q_k,\,q(t_{k+1})=q_{k+1} \rc,$
will be approximated by $\mathcal{C}^s([t_k,t_{k+1}],Q)\subset \mathcal{C}([t_k,t_{k+1}],Q)$, where $\mathcal{C}^s([t_k,t_{k+1}],Q)$ denotes the space of polynomials of degree $s$.  Given $s+1$ control points $0=d_0<d_1<\cdots<d_{s-1}<d_s=1$ and $s+1$ configurations $q_k=(q_k^0,q_k^1,\ldots,q_k^{s-1},q_k^s)$, with $q_k^0=q_k$ and $q_k^{s}=q_{k+1}$, then $q_d(t;k)\in\mathcal{C}^s([t_k,t_{k+1}],Q)$ can be defined by
\begin{equation}\label{Polynomials}
q_d(t;k)=\sum_{\nu=0}^sq_k^{\nu}\,\ell_{\nu}\left(\frac{t}{h}\right),
\end{equation}
where $\ell_{\nu}(\tau)$ are Lagrange polynomials of degree $s$ such that $\ell_{\nu}(d_i)= \delta_{\nu i}$ (here $\delta$ is the Kronecker symbol), and therefore $q_d(h\,d_i;k)=q_k^i$ according to \eqref{Polynomials}. Moreover, the time derivative of \eqref{Polynomials} is 
$$\dot q_d(t;k)=\frac{1}{h}\sum_{\nu=0}^sq_k^{\nu}\,\dot \ell_{\nu}\left(\frac{t}{h}\right).$$

\paragraph{(2) Quadrature for the action integral} For the approximation of the action integral \eqref{ContAc}, first we replace the curves $q(t)$ and $\dot q(t)$ by their polynomial counterparts $q_d(t;k)$, $\dot q_d(t;k)$ in the interval $[k\,h\,,\,(k+1)\,h]$, i.e.~
\[
\int_{kh}^{(k+1)h}L(q_d(t;k),\dot q_d(t;k))\,dt,\quad k=0,\ldots, N-1.
\]
Next, in the same time interval, a quadrature rule $(b_i,c_i)_{i=1}^r$ is applied, with $c_i\in[0,1]$. This defines the discrete Lagrangian:
\begin{equation}\label{DiscLag}
L_d(q_k)\equiv L_d(q_k^0, \ldots,q_k^s):=h\sum_{i=1}^rb_iL(q_d(c_i\,h;k),\dot q_d(c_i\,h;k));
\end{equation}
it is important to remark that $L_d$ depends on $s+1$ variables. Naturally, the choice of quadrature should be adapted to the desired order of approximation with respect to the continuous action in order that now can be arbitrarily high.

The construction of higher order  variational integrators can be summarized in Figure \ref{fig:interpolation}.
\begin{figure}[H]
    \centering
    \includegraphics[width=.7\textwidth]{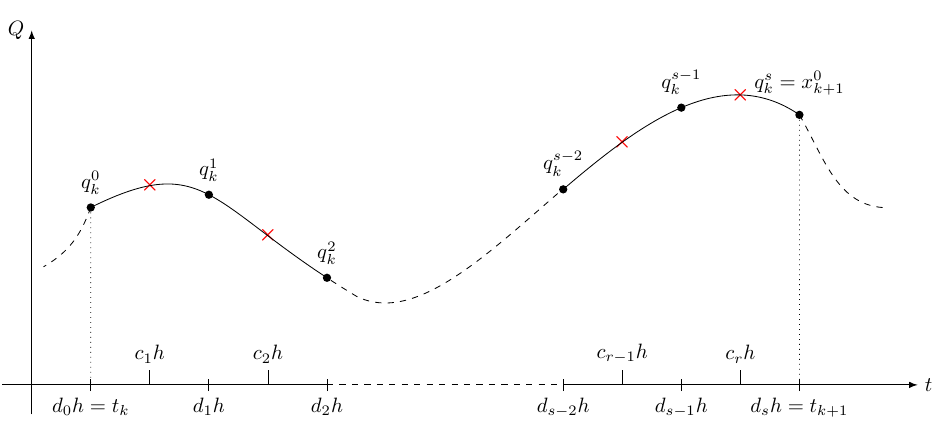}
    \caption{Polynomial interpolation principles. On each subinterval $[t_k,t_{k+1}]$, the trajectory interpolated by a polynomial passing through
the  points $\{q_k^i\}_{i=0}^{s}$ associated with the control points $\{d_\nu h\}$. The evaluations are made by the quadrature points for the
$\{c_ih\}$, indicated by  cross points.}
    \label{fig:interpolation}
\end{figure}

\medskip

Now, following \S\ref{Disc}, the action sum can be defined from \eqref{DiscLag} by
\[
\mathcal{S}_d(q_d)=\sum_{k=0}^{N-1}L_d(q_k),
\]
and the discrete Hamilton's principle can be applied in order to obtain a necessary and sufficient condition for its extremals, i.e.~ the discrete Euler-Lagrange equations. These are:
\[
\begin{split}
D_{s+1}L_d(q_{k-1}^0,\ldots,q_{k-1}^s)+D_1L_d(q_{k}^0,\ldots,q_{k}^s)=0,\\
D_iL_d(q_{k}^0,\ldots,q_{k}^s)=0,\quad  \forall\,i=2,\ldots,s;
\end{split}
\]
for $k=1,\ldots,N-1,$ where the transition condition, namely $q_{k-1}^s=q_k^0$ has to be taken into account. See \cite{MaWe, SinaSaake} for further details.

Higher order variational integrators (also denoted as Galerkin variational integrators) have been extensively studied in \cite{HaLe13,Leok2011,SinaSaake,Ca13} for conservative systems and in \cite{Ca14} for optimal control problems. 
Obviously, the convergence order of higher order variational integrators is limited by the order of the function space approximation and the order of the quadrature rule.
In \cite{HaLe13} a lower error bound is provided by the approximation order of the finite-dimensional function space (e.g.~convergence order $s$ is reached by using the space of polynomials of degree $s$).
Numerical convergence studies  in \cite{SinaSaake} indicate that Galerkin variational integrators based on the Lobatto and Gauss quadrature rules are of order $\min{(2s,u)}$, where $s$ is the degree of the polynomial and $u$ the order of the quadrature rule. A general proof of this superconvergence result is provided in \cite{SinaMats} using backward error analysis in the context of the calculus of variations. For particular classes that are equivalent to so-called (modified) symplectic Runge-Kutta methods, a proof of a  superconvergence has been made  in \cite{O14}.

\section{Fractional integrals, fractional derivatives and their approximations by discrete convolutions}\label{FracIntegrals}

\subsection{Fractional integrals and fractional derivatives}
Let us  start by giving a brief overview concerning   the fractional operators. We refer to \cite{TheBook,Oldham,TheBook2}  for more details.

\subsubsection{Fractional integrals}
The Riemann-Liouville $\alpha$-fractional integrals, Re $\alpha>0$, for $f:[0,T]\Flder\R$\footnote{For $f\in L^1([0,T];\R)$,   $J^{\alpha}_{-}f$ and $J^{\alpha}_{+}f$ are defined almost everywhere on $[0,T]$. Moreover,  they are defined everywhere on $[0,T]$, if  $f\in C([0,T];\R)$. (see \cite{Kai,Lavoie, Oldham} for more details).} are defined by
\begin{subequations}\label{RLInt}
\begin{align}
J^{\alpha}_{-}f(t)&=\frac{1}{\Gamma(\alpha)}\int_0^t(t-\tau)^{\alpha-1}f(\tau)\,d\tau,\,\,\,t\in (0,T],\label{RLInt:a}\\
J^{\alpha}_{+}f(t)&=\frac{1}{\Gamma(\alpha)}\int_t^T(\tau-t)^{\alpha-1}f(\tau)\,d\tau,\,\,\,t\in [0,T), \label{RLInt:b}
\end{align}
\end{subequations}
 where $\Gamma$ is the Euler Gamma function and we set $J^{0}_{-}f=J^{0}_{+}f=f$. The fractional integrals satisfy the so-called {\it semigroup property} \cite[Theorem 2.5, p.46]{TheBook}, i.e.~
\begin{equation}\label{Semigroup}
J_{\sigma}^{\alpha}\,J_{\sigma}^{\beta}f(t)=J_{\sigma}^{\alpha+\beta}f(t),\quad \sigma\in\lc -,+\rc.
\end{equation}

Let $m\in \N$. If the function $f$ is   continuously differentiable in $[0,T]$, then it can be continued analytically to  Re $\alpha<0$ via
\begin{equation}\label{GeneralDef}
J^{\alpha}_{-} f(t)=\frac{d^m}{dt^m}J^{m+\alpha}_{-}f(t), \quad J^{\alpha}_{+} f(t)=\frac{d^m}{dt^m}J^{m+\alpha}_{+}f(t) \quad \mbox{for Re} \ \alpha >-m.
\end{equation}
In this case, the fractional integral is called {\it fractional derivative}, and can be denoted by $D^{-\alpha}$ (note that the real part of $\alpha$ is now negative). 

In particular, from \eqref{GeneralDef}, and restricting ourselves to Re $\alpha\in [0,1]$ (which will be the range of interest in this work), we can establish the following definition of Riemann-Liouville fractional derivatives, which is usually found in the literature: 

\subsubsection{Fractional derivatives}
Let Re $\alpha\in [0,1]$. The left and right Riemann–Liouville fractional derivatives are respectively deﬁned by
\begin{equation}\label{RLDer}
\begin{split}
D^{\alpha}_{-}f(t)&=\,\,\,\,\frac{d}{dt}\,J^{1-\alpha}_{-}f(t)=\quad \frac{1}{\Gamma(1-\alpha)}\frac{d}{dt}\int_0^t(t-\tau)^{-\alpha}f(\tau)\,d\tau,\,\,\,t\in (0,T], \\
D^{\alpha}_{+}f(t)&=-\frac{d}{dt}\,J^{1-\alpha}_{+}f(t)=-\frac{1}{\Gamma(1-\alpha)}\frac{d}{dt}\int_t^T(\tau-t)^{-\alpha}f(\tau)\,d\tau,\,\,\,t\in [0,T),
\end{split}
\end{equation}
provided that $f\in AC([0,T],\R)$ which is  a very simple sufficient condition for the existence. 
 It is easy to see that $D^{0}_{-}f=D^{0}_{+}f=f$, whereas it can be proven that 
\begin{equation}\label{alpha1}
D^{1}_{-}f=-D^{1}_{+}f=df/dt.
\end{equation}
The last relationships follow easily from the definitions \eqref{RLDer} (first equality) and $J^{0}_{-}f=J^{0}_{+}f=f$, but the latter are not trivial from the definitions \eqref{RLInt} \footnote{The expression of the fractional derivative as a fractional integral of negative $\alpha$-index (and therefore its inverse operator) can be made explicit. Namely, from \eqref{Semigroup}, \eqref{RLDer} and \eqref{alpha1}, plus considering that $J_{-}^1$ is the usual integral operator (inverse of the derivative), we have
\[
D^{\alpha}_{-}f=D^1_{-}\,J_{-}^{1-\alpha}f=D^1_{-}\,J_{-}^{1}\,J_{-}^{-\alpha}f=J_{-}^{-\alpha}\,f,
\]
where, recall, $\alpha\in[0,1]$. A similar computation can be done for the Caputo definition of the fractional derivative \cite{TheBook}, i.e.~ $_cD^{\alpha}_{-}f:=J^{1-\alpha}_{-}\,D^1_{-}f=D^{\alpha}_{-}f$ by imposing that $f(0)=0$.}.\\

Other relevant properties of fractional derivatives are 
\begin{subequations}\label{FracProperties}
\begin{align}
\int_0^Tf(t)D^{\alpha}_{\lambda}g(t)\, dt&=\int_0^T g(t)\big[D^{\alpha}_{-\lambda}f(t)\big]    \,dt,\quad \sigma\in\lc-,+\rc, \label{IntegrationByParts}\\
D_{\lambda}^{\alpha}D_{\lambda}^{\beta}&=D_{\lambda}^{\alpha+\beta},\quad 0\leq\alpha,\beta\leq1/2.\label{Aditive}
\end{align}
\end{subequations}
Property \eqref{IntegrationByParts} is called ``asymmetric integration by parts'' of the fractional derivatives, whereas \eqref{Aditive} is called again the ``semigroup property''.

\begin{remark}
It is important to note that, contrary to the Caputo derivatives, the Riemann–Liouville derivatives of a function $f$, in particular $D_-^\alpha f(t)$,  could lead to singularity at $t=0$. That is why the Caputo derivatives are more useful in applications, otherwise   we should impose $f(0)=0$, which will be the case in our discussion in Section \ref{sec:RVP}.
\end{remark}

\subsection{Fractional integrals as convolutions}

For the time being, let us focus on the {\it retarded} fractional integral in \eqref{RLInt}, i.e.~ $J^{\alpha}_{-}f(t)$. It is easy to see that it is a particular convolution integral
\begin{equation}\label{ContConv}
J^{\alpha}_{-}f(t)=(\kappa^{(\alpha)}* f)(t):=\int_0^t\kappa^{(\alpha)}(t-\tau)f(\tau)d\tau
\end{equation}
where the convolution kernel is given by
\begin{equation}\label{ConvKer}
\kappa^{(\alpha)}(t)=\frac{t^{\alpha-1}}{\Gamma(\alpha)}.
\end{equation}
Here, $(\alpha)$ must not be understood as a power in the left hand side, it is just a superscript denoting the dependence of the kernel on the parameter $\alpha$.\footnote{In complete generality, a convolution integral can be defined for any kernel as
\begin{equation}\label{GeneralConvo}
(\kappa* f)(t):=\int_0^t\kappa(t-\tau)f(\tau)d\tau.
\end{equation}}. The Laplace transform of this convolution kernel is given by
\begin{equation}\label{LaplaceTrans}
K^{(\alpha)}(s):=\mathcal{L}(\kappa^{(\alpha)})(s)=\int_0^{\infty}\frac{t^{\alpha-1}}{\Gamma(\alpha)}e^{-st}dt=s^{-\alpha}.
\end{equation}

\subsection{Discrete convolution}\label{ConQua}
The theory of discrete convolutions is developed in \cite{Lubich1,Lubich2,Lubich3} by Ch. Lubich (indeed, the topic of \cite{Lubich1} is the discretization of fractional integrals).

As in \S\ref{Disc}, let us consider a time grid $t_k=\lc k\,h\,| k=0,\ldots, N\rc$, where $h\in\R_+$ is the time step and $h\,N=T$. Moreover, consider a discrete series  $\lc f_{k}\rc_{0:N}\in (\R^d)^{N+1}$, where $f_k$ shall be an approximation of $f(t_k)$.

Now, define the discrete convolution in the following as an approximation of the continuous convolution $(\kappa* f)(t_k)$ in \eqref{GeneralConvo}, namely

\begin{equation}\label{DiscConv}
(\kappa* f)(t_k)\approx(\omega*f)(t_k):=\sum_{n=0}^{\infty}\omega_nf_{k-n}=^{1}\sum_{n=0}^{k}\omega_nf_{k-n},
\end{equation}
 observe that the series is truncated after $=^{1}$ since $f_k$ is not defined for $k<0$ \footnote{It could be argued that a more natural expression for the discrete convolution would be 
\[
(\omega*f)(t_k):=\sum_{n=0}^{k}\omega_{k-n}^{}f_{n},
\]
according to \eqref{ContConv} and the relationship between the continuous and discrete times. It can be shown easily that both expressions are equivalent after rearranging the discrete time index and reordering the coefficients. Moreover, \eqref{DiscConv} will make more sense when the discrete convolution is understood as the application of a certain operator over the discrete series $\lc f_{k}\rc_{0:N}$}, with a general convolution kernel $\kappa$. The convolution quadrature weights $\omega_n$ are defined as the coefficients of the generating power series
\begin{equation}\label{SeriesK}
K\left(\frac{\gamma(z)}{h}\right):=\sum_{n=0}^{\infty}\omega_nz^n,\quad\quad |z|\quad \mbox{small}.
\end{equation}
Here, $K(s)$ is the Laplace transform of the kernel $\kappa$ and the so-called  characteristic function $\gamma(z)=\sum_{n=0}^{\infty}\gamma_nz^n$  is the quotient of the generating polynomials $(\rho,\sigma)$ of a linear multistep method (LMM) $$\rho_0y_k+\rho_1 y_{k-1}+\ldots+\rho_ny_{k-n}=h(\sigma_0f_k+\sigma_1 f_{k-1}+\ldots+\sigma_nf_{k-n}) $$ 
for the differential equation $y^{\prime}=f(y)$, i.e.~
\begin{equation}\label{DeltaDef}
\gamma(z)=\frac{\rho(z)}{\sigma(z)}=\frac{\rho_0+\rho_1 z+\ldots+\rho_nz^n}{\sigma_0+\sigma_1 z+\ldots+\sigma_nz^n},
\end{equation}
where we assume that $\rho_0/\sigma_0>0$,  so that \eqref{SeriesK} is well-defined at least for sufficiently small $h$. Note that 
if we define $z_{-}$ to be the discrete backward operator, i.e.~~
\[
z_{-}\,f_k=f_{k-1},
\]
then the LMM can be defined as $\rho(z_{-})\,y_k=h\,\sigma(z_{-})\,f_k$, where $\lc y_k\rc_{0:N}\in(\R^d)^{N+1}$ is also a discrete series.

More importantly for the purposes of this article, the discrete convolution approximating the fractional integral \eqref{RLInt:a},\eqref{ContConv} can be redefined, according to \eqref{SeriesK}, by
\begin{equation}\label{DiscIntMinus}
\mathcal{J}_{-}^{\alpha}f_k:=K^{(\alpha)}\left(\frac{\gamma(z_{-})}{h}\right)\,f_k=\sum_{n=0}^{k}\omega_n^{(\alpha)}f_{k-n},
\end{equation}
where $K^{(\alpha)}(\gamma(z)/h)$, with $K^{(\alpha)}$ given in \eqref{LaplaceTrans}, has to be understood as an operator acting on $\lc f_k\rc_{0:N}$.  On the other hand,  considering the discrete forward operator
\[
z_{+}\,f_k=f_{k+1},
\]
the discrete convolution approximating the fractional integral \eqref{RLInt:b} will be
\begin{equation}\label{DiscIntPlus}
\mathcal{J}_{+}^{\alpha}f_k:=K^{(\alpha)}\left(\frac{\gamma(z_{+})}{h}\right)\,f_k=\sum_{n=0}^{N-k}\omega_n^{(\alpha)}f_{k+n},
\end{equation}
where again the series \eqref{SeriesK} gets truncated because $\lc f_k\rc_{0:N}$ is undefined for $k>N$\footnote{For a given function, the discrete convolution can be also defined in continuous time, namely
\[
\mathcal{J}_{\lambda}^{\alpha}f(t):=K^{(\alpha)}(\gamma(z_{\lambda})/h)\,f(t)=\sum_{n\geq0}\omega_n^{(\alpha)}f(t\,\pm\,n\,h).
\]
}. It is interesting to note that, following these definitions, the convolution weights $\omega^{(\alpha)}_n$ are the same for $\mathcal{J}^{\alpha}_{-}$ and $\mathcal{J}^{\alpha}_{+}$.

Now we are in situation to show the proof of some properties relevant for future results.

\begin{lemma}\label{ConvProperties}
Consider two discrete series $\lc f_k\rc_{\tiny 0:N}, \lc g_k\rc_{\tiny 0:N}$. Then the following properties hold true:

\begin{itemize}
\item[(1)] The semigroup property of the discrete convolution: $\mathcal{J}^{\alpha}_{\lambda}\,\mathcal{J}^{\beta}_{\lambda}\,f_k=\mathcal{J}^{\alpha+\beta}_{\lambda}\,f_k.$

\item[(2)] The asymmetric integration by parts:
\begin{equation}\label{AsymmetricInt:1}
\sum_{k=0}^{N}g_k\,(\mathcal{J}_{-}^{\alpha}f_k)=\sum_{k=0}^N(\mathcal{J}_{+}^{\alpha}g_k)\,f_k.
\end{equation}
\end{itemize}
\end{lemma}

\begin{proof}

(1) From the definitions \eqref{DiscIntMinus} and \eqref{DiscIntPlus} (first equalities) with $K^{(\alpha)}(s)=s^{-\alpha}$ \eqref{LaplaceTrans}, it follows:
\begin{align*}
    \mathcal{J}^{\alpha}_{\lambda}\,\mathcal{J}^{\beta}_{\lambda}\,f_k&=K^{(\alpha)}\left(\frac{\gamma(z_{\lambda})}{h}\right)K^{(\beta)}\left(\frac{\gamma(z_{\lambda})}{h}\right)\,f_k\\
    &=\left(\frac{\gamma(z_{\lambda})}{h}\right)^{-\alpha}\left(\frac{\gamma(z_{\lambda})}{h}\right)^{-\beta}\,f_k\\
&=\left(\frac{\gamma(z_{\lambda})}{h}\right)^{-\alpha-\beta}
=K^{\alpha+\beta}\left(\frac{\gamma(z_{\lambda})}{h}\right)
=\mathcal{J}^{\alpha+\beta}_{\lambda}\,f_k,
\end{align*}

(2) See \cite{Cresson1} for more details:

\[
\begin{split}
\sum_{k=0}^Ng_k\,(\mathcal{J}_{-}^{\alpha}f_k)&=^{_1}\sum_{k=0}^N\sum_{n=0}^k\omega_n^{(\alpha)}\,g_k\,f_{k-n}=^{_2}\sum_{n=0}^N\sum_{k=n}^N\omega_n^{(\alpha)}\,g_k\,f_{k-n}\\
&=^{_3}\sum_{n=0}^N\sum_{k=0}^{N-n}\omega_n^{(\alpha)}\,g_{k+n}\,f_{k}=^{_4}\sum_{k=0}^N\sum_{n=0}^{N-k}\omega_n^{(\alpha)}\,g_{k+n}\,f_{k}=^{_6}\sum_{k=0}^{N}(\mathcal{J}_{+}^{\alpha}g_k)\,f_k.
\end{split}
\]
In $=^{_1}$ the definition \eqref{DiscIntMinus} (second equality) is used. To prove $=^{_2}$  is enough to notice that, for fixed $j=0, \ldots,N,$  the elements  $a_{i}:=\omega_i^{(\alpha)}\,g_j\,f_{j-i}$, $i=0, \ldots,j$, on the left hand side, disposed in columns, form an upper diagonal $(N+1)\times (N+1)$, whereas the same elements on the right hand side, for $j=0, \ldots,N$ and $i=j, \ldots,N$, account for the transposed matrix; therefore their total sums are equal. In $=^{_3}$ the sum index is rearranged. In $=^{_4}$ equivalent arguments to $=^{_2}$ can be used. Finally, in $=^{_6}$ the definition \eqref{DiscIntPlus} (second equality) is used; this concludes the proof. 
\end{proof}

Now, we take into consideration the convergence order of $\mathcal{J}_{\lambda}^{\alpha}$ with respect to $J^{\alpha}_{\lambda}$ following \cite{Lubich1}; as we will see, this order is sensitive to the convergence order of the multistep method $(\rho, \sigma)$. In particular, for the generating function $\gamma=\rho/\sigma$, which gives rise to $\omega_n^{(\alpha)}$ according to \eqref{DiscIntMinus}, we say that the corresponding multistep method is convergent of order $p$ if and only if 
\begin{subequations}
\label{ConvergenceDelta}
\begin{equation}
    \tilde\gamma_n\quad\mbox{are bounded},\,\,\mbox{(stability)}\label{ConvergenceDelta1}
\end{equation}
\begin{equation}
    h\,\tilde\gamma (\,e^{-h}\,)=1+O(h^p),\quad\mbox{as}\,\,h\Flder 0,\,\mbox{(order $p$ consistency)},\label{ConvergenceDelta2}
\end{equation}
\end{subequations}
where $\tilde\gamma_n$ are coefficients of  the power series $\gamma^{-1}(z):=1/\gamma(z)$. 
Now, we introduce the notion of convergence of the quadrature $\omega_n^{(\alpha)}$.
\begin{definition}\label{DefConver}
A convolution quadrature determined by the coefficients $\omega_n^{(\alpha)}$ is convergent of order $p$ (to $J^{\alpha}_{\lambda}$) if
\begin{equation}\label{OrderSat}
J^{\alpha}_{\lambda }t^{\beta-1}-\mathcal{J}^{\alpha}_{\lambda}t^{\beta-1}=O(h^{\beta})+O(h^p),
\end{equation}
for all $\beta\in\C$, $\beta\neq 0,-1,-2,\ldots.$
\end{definition}

\begin{theorem}[Theorem 2.6 in \cite{Lubich1}]\label{LubichsTheo}
Let $(\rho,\sigma)$ denote an implicit linear multistep method which is convergent of order $p$ \eqref{ConvergenceDelta} and assume that the zeros of $\sigma$ have absolute value less than 1. Then, $\mathcal{J}^{\alpha}_{\lambda}$ \eqref{DiscIntMinus}, \eqref{DiscIntPlus} are convergent of order $p$ (Definition \ref{DefConver}) to $J^{\alpha}_{\lambda}$ \eqref{RLInt}.
\end{theorem}

In \cite{Lubich1} is also established, under the condition of $p$-convergence in Definition \ref{DefConver}, that for functions $f(t)=t^{\beta-1}g(t)$, $g(t)$ smooth, there always exists a starting quadrature $\varpi_{k,n}$\footnote{In practice, computing  $\varpi_{k,n}$ becomes more complicated for some values of $\alpha$ where a linear system should be solved at each step.} which is defined by
\begin{equation}\label{StartingQuad}
\tilde{\mathcal{J}}_{-}^{\alpha}f_k:=\sum_{n=0}^{k}\omega_{n}^{(\alpha)}f_{k-n}+h^{\alpha}\sum_{n=0}^s\varpi_{k,n}f_n,
\end{equation}
with $s$ fixed, such that
\begin{equation}\label{OrderMax}
J^{\alpha}_{\lambda}f-\tilde{\mathcal{J}}_{\lambda}^{\alpha}f=O(h^p),
\end{equation}
uniformly for $t\in[0,T]$. In other words, to achieve the order $p$  convergence of the underlying LMM,  the extra term in \eqref{OrderMax} should be introduced in order to eliminate low order terms in the error bound in \eqref{OrderSat}.

It is important to remark that due to the presence of the extra $\varpi$-initial terms in \eqref{StartingQuad} the {\it convolution structure} is violated, and therefore the properties proved in Lemma \ref{ConvProperties} are no longer true for $\tilde{\mathcal{J}}_{\lambda}^{\alpha}f$. In Lubich's own words: ``{\it The convolution structure is only violated by the few correction terms of the starting quadrature which will be necessary for high order schemes}''. Indeed, this is also relevant in terms of the convergence order, since from Definition \ref{DefConver}, Theorem \ref{LubichsTheo} and \eqref{OrderMax} we conclude that for a function $f(t)=t^{\beta-1}g(t)$, with $g(t)$ smooth, a multistep method $(\rho,\sigma)$ is $p$-convergent would generate the following convergence bound:
\begin{equation}\label{SaturationOrder}
J^{\alpha}_{\lambda}f-\mathcal{J}_{\lambda}^{\alpha}f=O(h^{\beta})+O(h^p).
\end{equation}
Thus, we expect the saturation of the convergence order at $\min(\beta,p)$. We illustrate the saturation \eqref{SaturationOrder} for the functions $t\mapsto t^{\beta-1}\sin(t),\ \beta = 1,3,4,5$ in Figure \ref{Saturation}, which is a log-plot of $h$ versus the error $e_h$, defined as usual in the literature by means of the maximum norm, i.e.~
\begin{equation}\label{eq:caputoerror}
    e_h= \max_{0\leq k\leq N}\,\big|J^{\alpha}_{-}f(t_k)-\mathcal{J}_{-}^{\alpha}f_k\big|,	
\end{equation}
where the $h$-dependence of $e_h$ is implicit in the time grid $t_k$. 
Of our interested,  the characteristic function  $\gamma(z_{-})$ of  the classical backward diﬀerentiation formula (BDF)  up to order $6$, that is,
\begin{equation}\label{BDFgenfunction}
\gamma(z)=\sum_{k=1}^p \frac{1}{k}(1-z)^k:=\gamma_p(z).
\end{equation}

\begin{figure}[H]
\centering
\includegraphics[width=.48\textwidth]{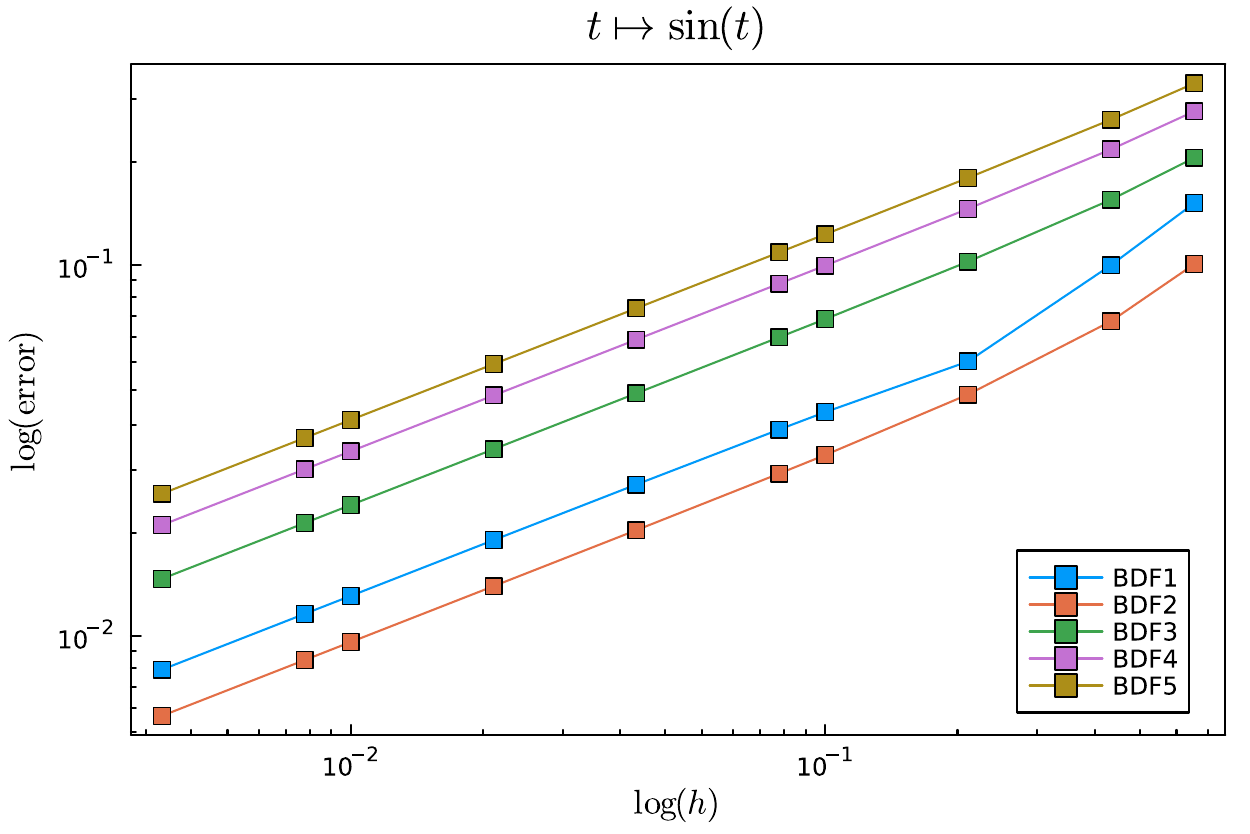}\hfill
\includegraphics[width=.48\textwidth]{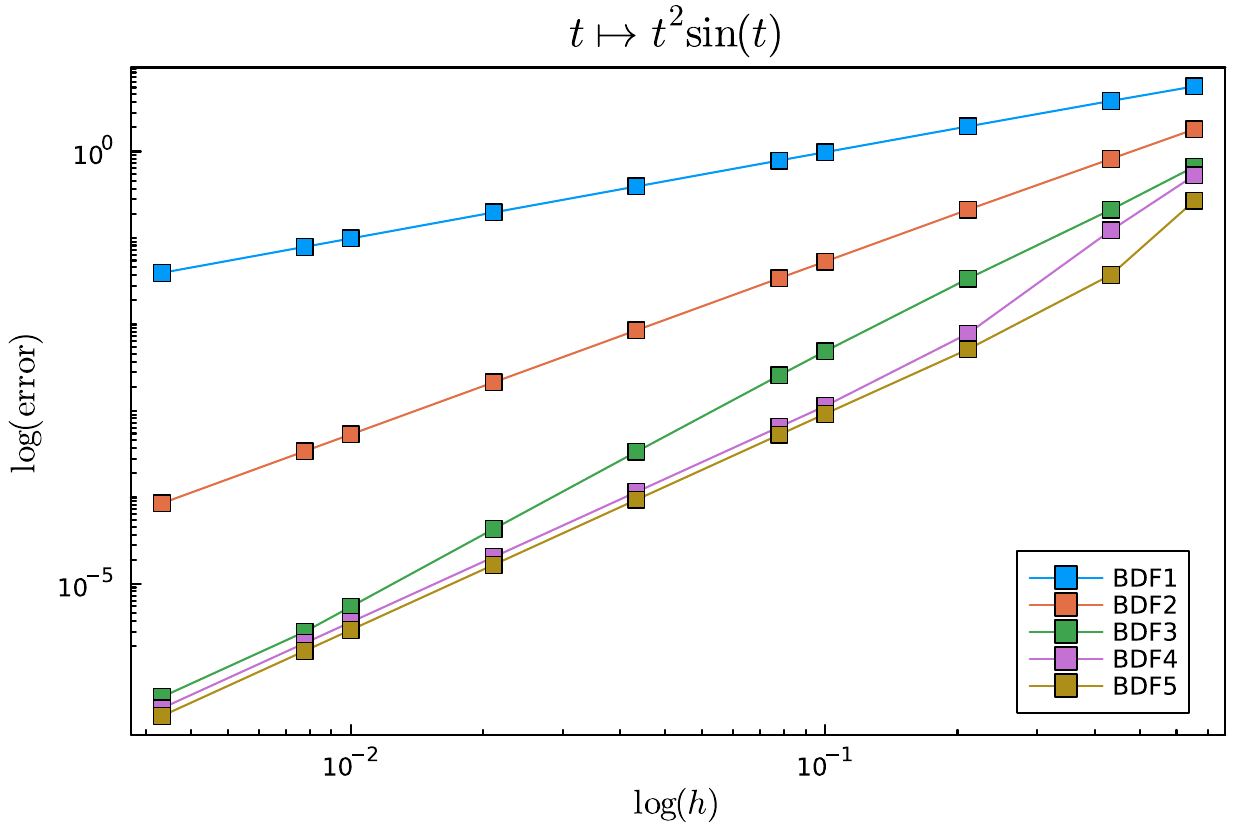}\\
\includegraphics[width=.48\textwidth]{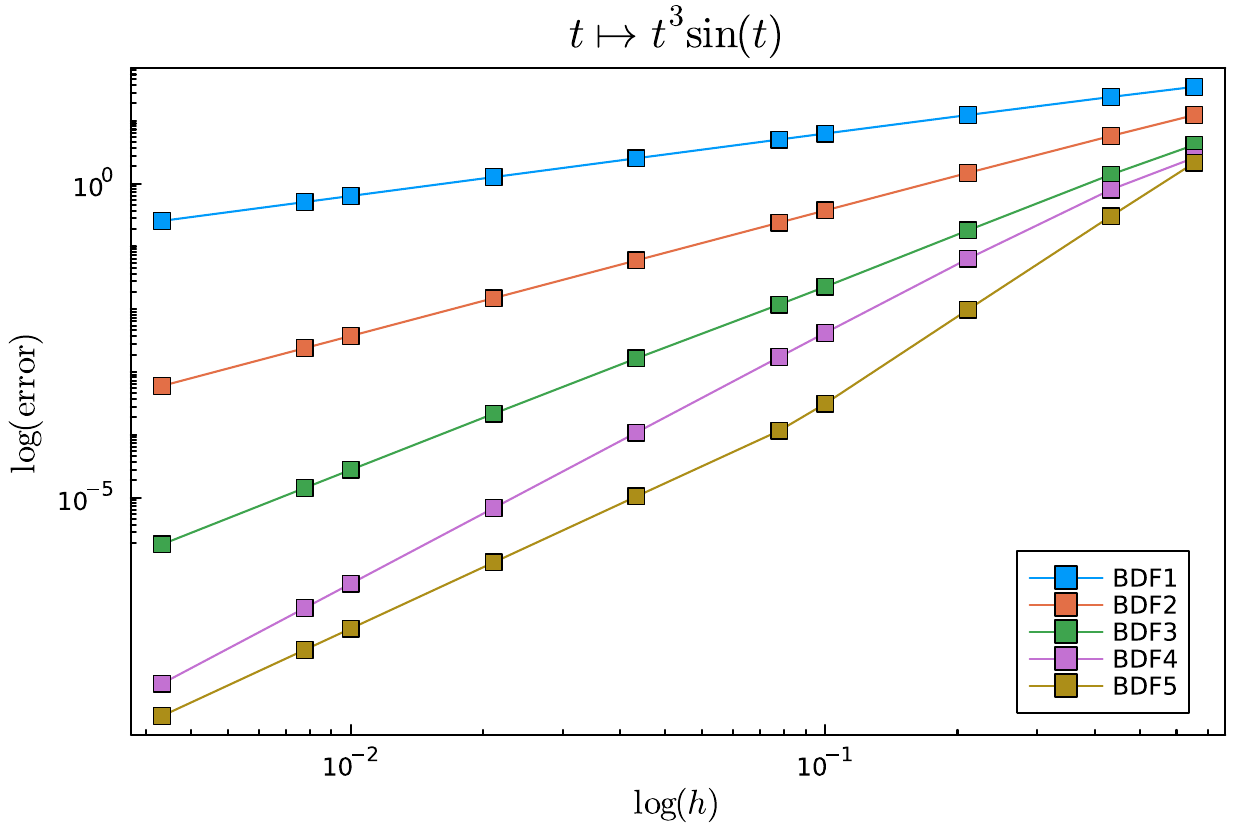}\hfill
\includegraphics[width=.48\textwidth]{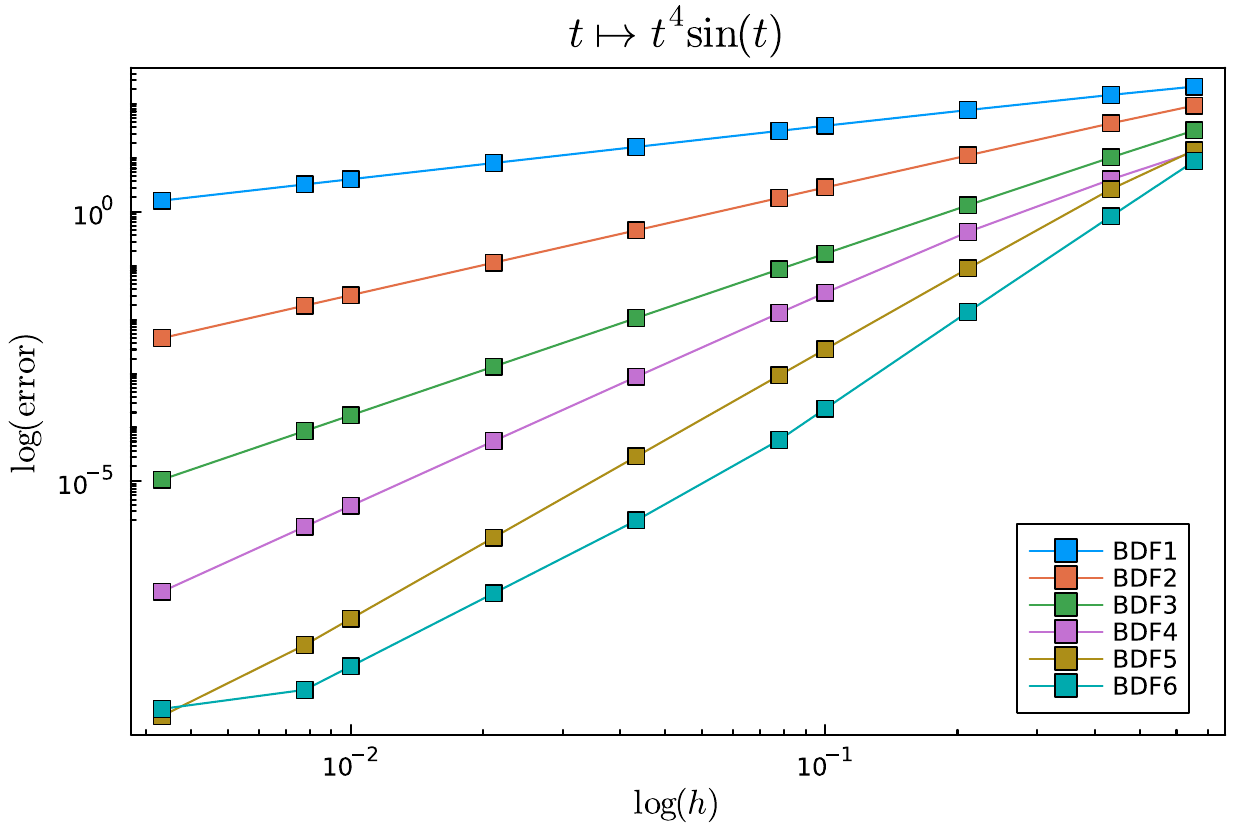}\hfill
\caption{Log-Log plot the error $e(h)$ versus  $h$  corresponds to the Caputo fractional derivative of order $1/2$ ($\alpha=-1/2$ in \eqref{eq:caputoerror}).  As expected, the convergence starts to saturate at $p=1$ (upper-left), $p=3$ (upper-right),  $p=4$ (lower-left) and $p=5$ (lower-right). }
\label{Saturation}
\end{figure}

\section{Restricted variational principle for the dynamics of Lagrangian systems subject to fractional damping}\label{sec:RVP}

In \cite{JiOb1,JiOb2} a restricted variational principle (both continuous and discrete) in order to obtain the dynamics of a Lagrangian system subject to fractional damping is delivered. As in other previous approaches treating dissipative systems in a variational fashion \cite{Agra,Bateman,Cresson,Galley,Riewe}, it is based on the doubling of variables and the introduction of their fractional derivatives (which in this work will be considered as fractional integrals with negative $\alpha$ index, as explained in \S\ref{FracIntegrals}) in the state space of the relevant Lagrangians.

\subsection{Continuous setting}

Let us consider the $AC^2$ curves $x,y:[0,T]\Flder\R^d$ and a $C^2$ Lagrangian function $L:T\R^d\Flder \R$ (also $L:\R^d\times\R^d\Flder\R$). Define a $C^2$ Lagrangian function\footnote{In \cite{JiOb1,JiOb2} the state space of $\mathcal{L}$ is defined as a particular vector bundle with base point $(x,y)$ and fiber $(\dot x,\dot y, J^{-\alpha}_{-}x, J^{-\beta}_{+}y)$. However, this space is totally isomorphic to the Cartesian product of six copies of $\R^d$ and we will stick to this for simplicity, preserving the notation of the base point $(x,y)$ as argument of the forthcoming actions also for simplicity.} 
\begin{equation}\label{GenerLag}
\begin{array}{rcl} \mathcal{L}:\R^d\times\R^d\times\R^d\times\R^d\times\R^d\times\R^d &  \Flder  & \R \\
 (x,\,y,\,\dot x,\,\dot y,\, J^{-\alpha}_{-}x,\, J^{-\beta}_{+}y)    & \mapsto&\mathcal{L}(x,y,\dot x,\dot y, J^{-\alpha}_{-}x,J^{-\beta}_{+}y)\\ [2ex]
  \mathcal{L} (x,\,y,\,\dot x,\,\dot y,\, J^{-\alpha}_{-}x,\, J^{-\beta}_{+}y)=L(x,\dot x)+&&\hspace{-1cm}L(y,\dot y)-\mu\,J^{-\alpha}_{-}x\,J^{-\beta}_{+}y,
\end{array}
\end{equation}
where $\alpha,\beta \in [0,1/2]$ and $\mu\in\R_+$. Given this particular Lagrangian, we define the relevant action:
\begin{equation}\label{FracAction}
\begin{split}
\Sg(x,y)=\Sc(x,y)&+\Sf(x,y), \\
\Sc(x,y)=\int_0^T(L(x(t),\dot x(t))+L(y(t),\dot y(t))\,)\, dt,&\quad \Sf(x,y)=-\mu\,\int_0^TJ^{-\alpha}_{-}x(t)\,J^{-\beta}_{+}y(t)\,dt,
\end{split}
\end{equation}
where {\it cons} goes after ``conservative'' and {\it frac} after ``fractional''. Moreover, let us define restricted varied curves by means of
\begin{equation}\label{VariedCurves}
x_{\epsilon}(t)=x(t)+\epsilon\, \delta x(t),\,\,\,y_{\epsilon}(t)=y(t)+\epsilon\, \delta x(t),
\end{equation}
with $\epsilon\in\R_{+}$ and an $AC^2$ $\,\,\delta x:[0,T]\Flder\R^d$ such that $\delta x(0)=\delta x(T)=0$ (observe that the variations are equal for both curves, which is the base of the {\it restriction}). With these elements, and assuming fixed endpoints $x(0),\,x(T),\,y(0),\,y(T)$, we can establish the restricted variational principle:
\begin{theorem}\label{ContTheo}
The equations
\begin{subequations}\label{ContFracDamp}
\begin{align}
\frac{d}{dt}\lp\frac{\der L}{\der\dot x}\rp-\frac{\der L}{\der x}=-\mu\,J^{-(\alpha+\beta)}_{-}x,\label{ContFracDamp:a}\\
\frac{d}{dt}\lp\frac{\der L}{\der\dot y}\rp-\frac{\der L}{\der y}=-\mu\,J^{-(\alpha+\beta)}_{+}y,\label{ContFracDamp:b}
\end{align}
\end{subequations}
are  {\rm sufficient} conditions for the extremals of $\Sg(x,y)$ \eqref{FracAction} under restricted calculus of variations \eqref{VariedCurves}.
\end{theorem}
It is important to remark that in the proof of this theorem (see \cite{JiOb2}), it is crucial the use of the asymmetric integration by parts \eqref{IntegrationByParts} and the semigroup property \eqref{Aditive} of the fractional derivatives. In addition, it is also proven in \cite[Proposition 3.2]{JiOb2} that under even parity of $L$ in the velocity variable, then \eqref{ContFracDamp:b}  reduces to \eqref{ContFracDamp:a}  in reversed time, i.e.~  $y(t)=x(T-t)$. Finally, it is easy to see that the dynamics \eqref{ContFracDamp}, say the Lagrangian dynamics subject to fractional damping, reduces to the usual linear damped dynamics when $\alpha=\beta=1/2$, according to \eqref{alpha1}.

In the following, we give several intermediary lemmas which are obtained by
generalizing the results presented in

\subsection{Discrete setting based on CQ}

In \cite{JiOb1,JiOb2}, the author applied the discretization procedure described in \S\ref{Disc} and the Grünwald-Letnikov approximation  for the fractional derivative to derive the so-called {\it fractional variational integrators}. In the following, we propose a generalization of this process using CQ. For that, let us consider two discrete series $x_d=\lc x_k\rc_{0:N}$ and $y_d=\lc y_k\rc_{0:N}$, as well as 
two particular discretizations  of \eqref{RLDer}, i.e.~
\begin{equation}\label{Order1CQ}
\mathcal{J}_{-}^{-\alpha}x_k=\sum_{n=0}^k\omega_n^{(-\alpha)}x_{k-n}, \qquad \mathcal{J}_{+}^{-\beta}y_k=\sum_{n=0}^{N-k}\omega_n^{(-\beta)}y_{k+n},
\end{equation}
where the weights $\omega_n^{(-\alpha)}$
are the  coefficients of the generating power series of $K^{(-\alpha)}(\gamma(z)/h)$ with $K^{(\alpha)}$ defined in \eqref{LaplaceTrans}, namely
$$K^{(-\alpha)}\left(\frac{\gamma(z)}{h}\right)=\left(\frac{\gamma(z)}{h}\right)^\alpha=\sum_{k=0}^\infty \omega_n^{(-\alpha)} z^n.$$
For concreteness,  we choose again $\gamma_p(z)$ as in  \eqref{BDFgenfunction}, the  characteristic  function of the  
backward differentiation formulas. The Grünwald weights used in   \cite{JiOb2}  is equivalent to CQ with a particular choice of $\gamma_p(z)$, i.e.~
$\gamma_1(z)=1-z$ and  the notation  $\mathcal{J}_{-}^{-\alpha}x_k$ and $\mathcal{J}_{+}^{-\beta}y_k$ have been used  as $\Delta^{\alpha}_{-}x_k$ and $\Delta^{\beta}_{+}y_k$ respectively.\\

 The discrete action, counterpart of \eqref{FracAction}, is then
\begin{equation}\label{DiscFracAction}
\begin{split}
\mathcal{L}_d(x_d,y_d)=\Sc_d(x_d,y_d)&+\Sf_d(x_d,y_d), \\
\Sc_d(x_d,y_d)=\sum_{k=0}^{N-1}(L_d(x_k,x_{k+1})+L_d(y_k,y_{k+1})),&\quad \Sf_d(x_d,y_d)=-\mu \,h\,\,\sum_{k=0}^{N}\mathcal{J}_{-}^{-\alpha}x_k \,\mathcal{J}_{+}^{-\beta}y_k.
\end{split}
\end{equation}  
Taking the discrete equivalent of the restricted varied curves, i.e.~
\begin{equation}\label{DiscVariedCurves}
x_d^{\epsilon}=x_d+\epsilon\, \lc\delta x_k\rc_{0:N},\,\,\,y_d^{\epsilon}=y_d+\epsilon\, \lc\delta x_k\rc_{0:N},
\end{equation}
such that $\delta x_0=\delta x_N=0$, we can establish the discrete counterpart of Theorem \eqref{ContTheo}:
\begin{theorem}\label{DiscTheo}
The equations
\begin{subequations}\label{DiscFracDamp}
\begin{align}
D_1L_d(x_k,x_{k+1})&+D_2L_d(x_{k-1},x_{k})=\mu\,h\,\mathcal{J}^{-(\alpha+\beta)}_{-}x_{k},\label{DiscFracDamp:a}\\
D_1L_d(y_k,y_{k+1})&+D_2L_d(y_{k-1},y_{k})=\mu\,h\,\mathcal{J}^{-(\alpha+\beta)}_{+}y_{k},\label{DiscFracDamp:b}
\end{align}
\end{subequations}
both for $k=1,\ldots,N-1$, are {\rm sufficient} conditions for the extremals of $\mathcal{L}_d(x_d,y_d)$ \eqref{DiscFracAction} under restricted calculus of variations \eqref{DiscVariedCurves}.
\end{theorem}

See \cite{JiOb2} for the proof. As in the continuous side, the semigroup property and asymmetric integration by parts of the discrete fractional derivatives/integrals,   properties (1) and (2) in Lemma \ref{ConvProperties}, respectively, are crucial in the proof of this theorem. It is also true that \eqref{DiscFracDamp:b} reduces to \eqref{DiscFracDamp:a} in discrete reversed time, i.e.~ $y_k=x_{N-k}$ (\cite{JiOb2}, Proposition 4.2). Naturally, the equations \eqref{DiscFracDamp} provide a discrete iteration scheme for the fractional damped dynamics \eqref{ContFracDamp}, with $\gamma_1(z)=1-z$, delivering a $p=1$ convergent integrator; see \cite{JiOb2} for further details.\\

In the following section, we will apply the order two midpoint variational integrator for the conservative part \cite{MaWe} and BDFCQ for the fractional one. Then, we will denote by FVI-BDFCQ  the scheme \eqref{DiscFracDamp:a} and  write it simply FVI when no confusion can arise. All the experiments are carried out in Julia Version 1.9.3.

\subsection{Numerical experiment} \label{subsec:experiment}
For the next example, we take the Lagrangian associated to the harmonic oscillator of the form $L(x,\dot x)=\dot x^2/2-x^2/2$  so that  the equation \eqref{ContFracDamp:a} reads for $\alpha=\beta$ 
\begin{equation}\label{eq:damped-oscillator}
    \ddot x+\mu\,D_-^{(2\alpha)} x+x=0, 
\end{equation}
where  $D_-^{(2\alpha)}\equiv J_-^{-(2\alpha)}$. Let us consider the problem on  $[0, 16]$ and choosing  $\alpha=1/2$ so that  fractional operator $D_-^{(2\alpha)}$ coincides with the usual operator $d/dt$. In this test run  we take the initial values $x(0)=0,\ \dot x(0)=1.2$ and $\mu=0.2$.  Firstly  we plot the numerical solution obtained by the FVI-BDF1CQ, explicit and implicit Euler integrators on the interval $[0, 16]$ with the stepsize $h = 0.125$ in Figure \ref{fig:FVIcovergence-integer} (left). The corresponding results of the energy dissipation  and the absolute errors are presented  in  Figure \ref{fig:FVIcovergence-integer} (right) and Figure \ref{fig:FVIcovergence-integer2} (left), respectively. Secondly we  compute the global errors as the maximum norm between the numerical solution and the exact solution, i.e. 
$$\max | x(t_k) - x_k|,\quad \forall k.$$
The results are presented in Figure \ref{fig:FVIcovergence-integer2} (right) as the global errors (in logarithmic scale) against stepsizes on $[0,16]$ with $h =16/2^i,\ i=4,\ldots,11$.

\begin{figure}[H]
\centering
\includegraphics[width=.48\textwidth]{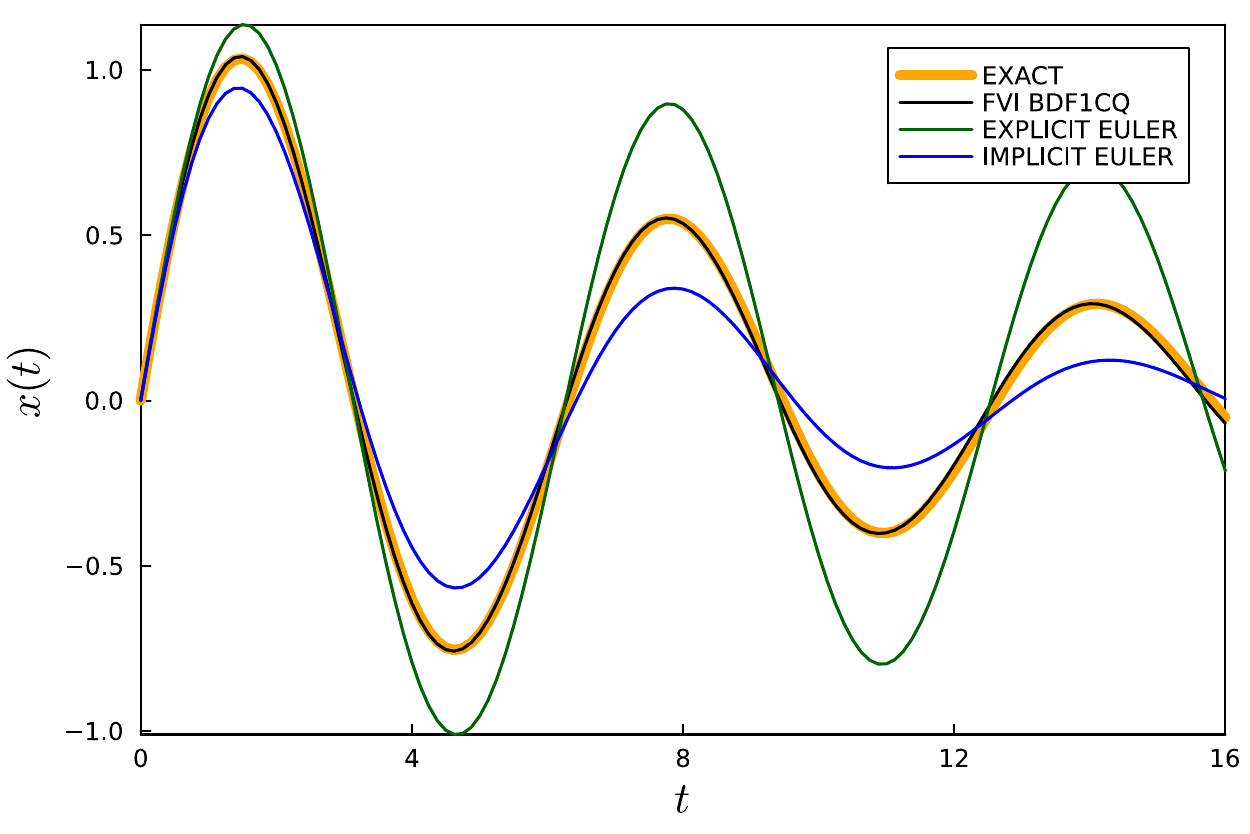}
\hfill
\includegraphics[width=.48\textwidth]{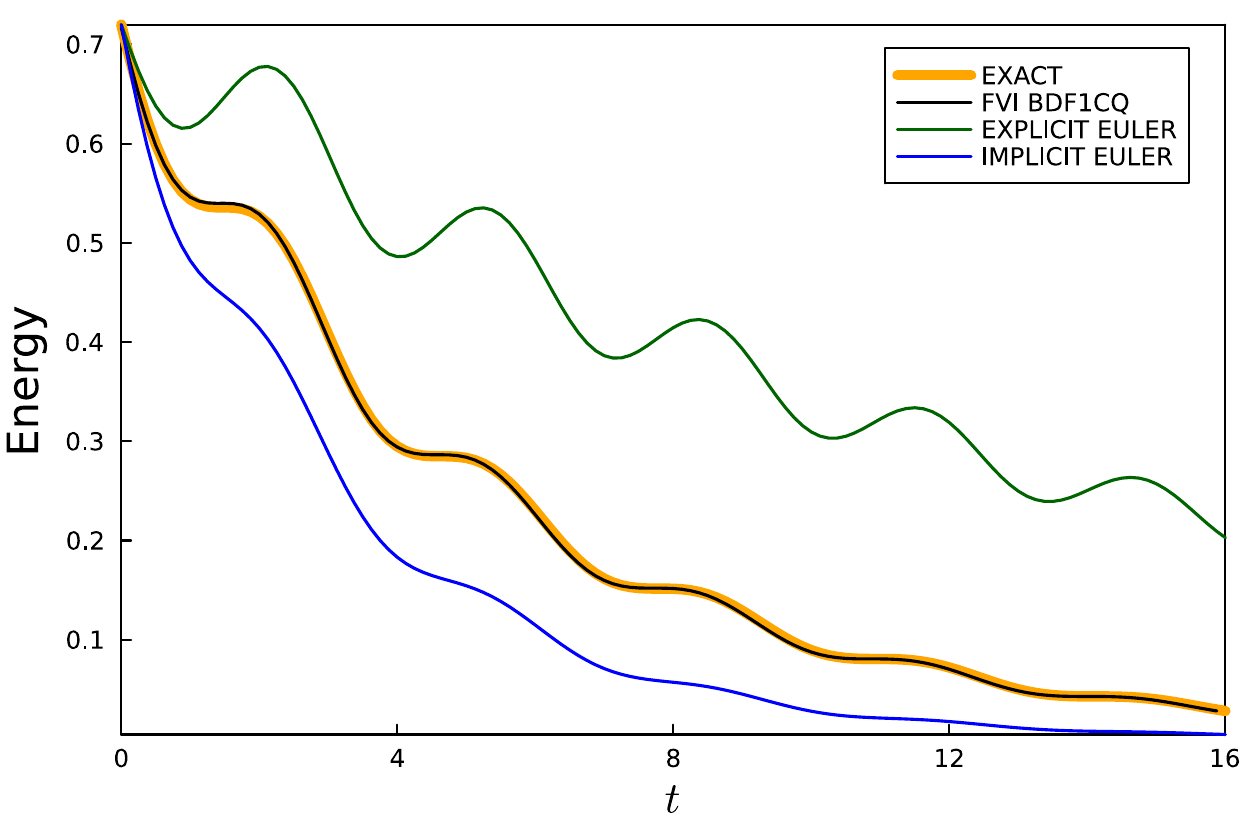}
\caption{Damped harmonic oscillator \eqref{eq:damped-oscillator} ($\alpha=1/2$). Left: Exact solution vs FVI-BDF1CQ method for $h=0.125$. Right: 
Energy behaviour for $h=0.125$.}
\label{fig:FVIcovergence-integer}
\end{figure}

The main property of a dissipative system is that the energy is always dissipated  with time and as we can see in Figure \ref{fig:FVIcovergence-integer} (right), FVI-BDF1CQ can preserve the dissipation structure of the damped harmonic oscillator which confirm that FVI-BDF1CQ gives good numerical behaviour.

\begin{figure}[H]
\centering
\includegraphics[width=.48\textwidth]{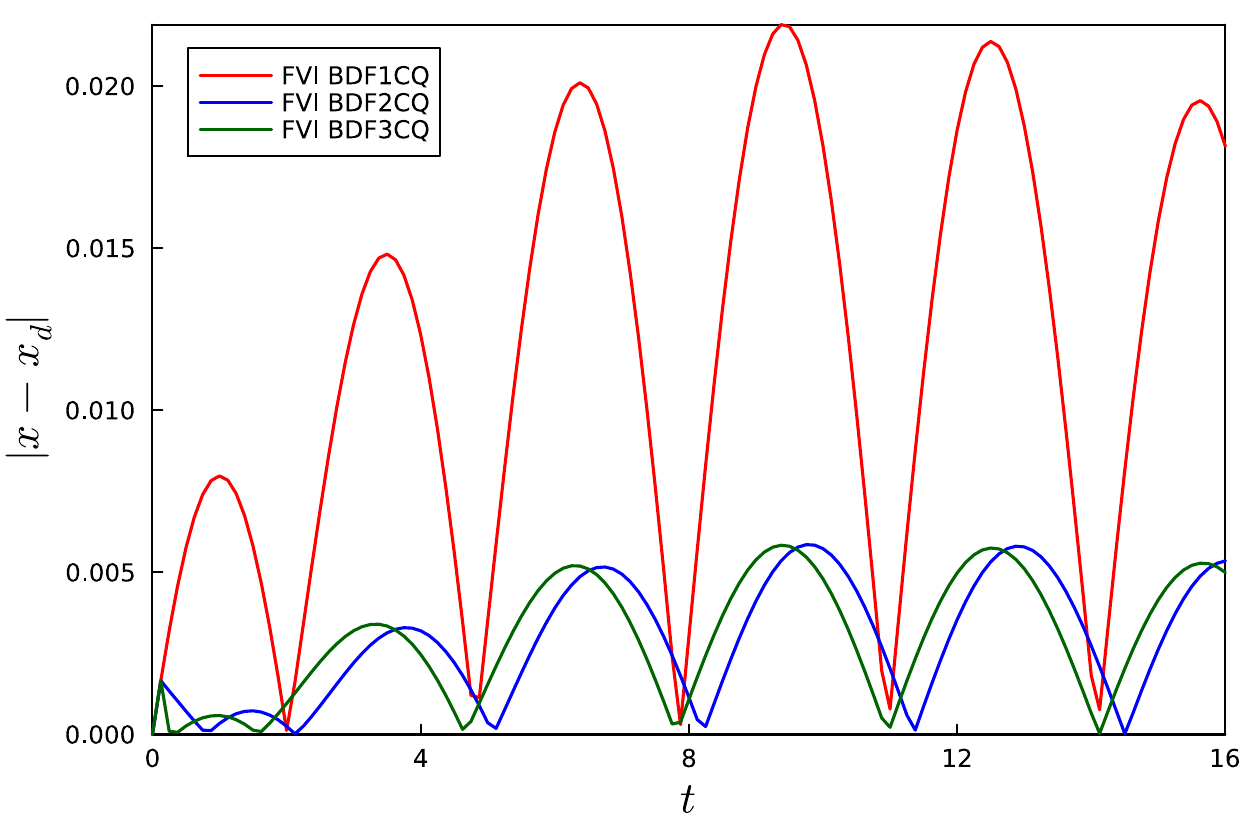}
\hfill
\includegraphics[width=.48\textwidth]{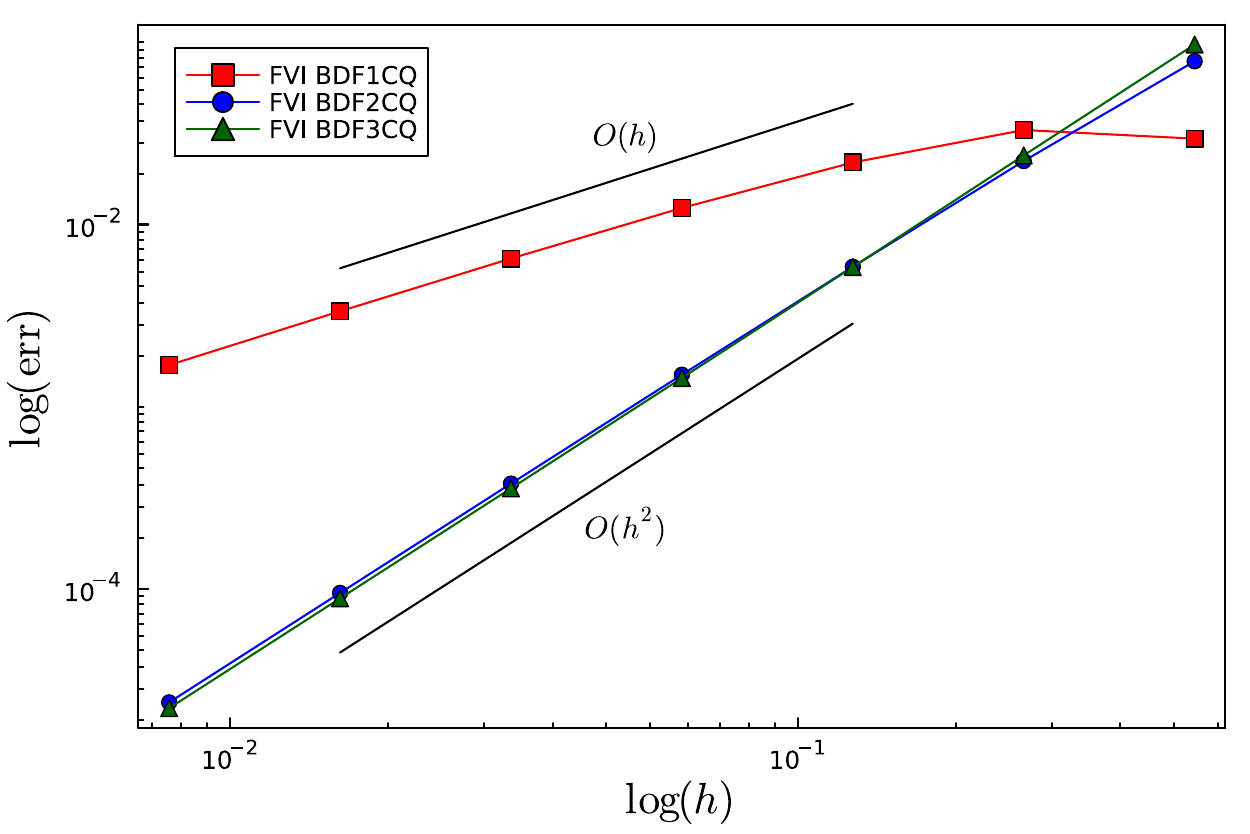}
\caption{Damped harmonic oscillator \eqref{eq:damped-oscillator}. Left: absolute errors  for $h=0.125$. Right: Log-Log plot of the global error presented on $t\in [0,16]$ for $h =16/2^i,\ i=4,\ldots,11$.}
\label{fig:FVIcovergence-integer2}
\end{figure}

We can confirm from Figure \ref{fig:FVIcovergence-integer2} (right) that the order of FVI-BDF1CQ is one and this result has been discussed in \cite{JiOb2}. However, we observe that the second order convergence both FVI-BDF2CQ and FVI-BDF3CQ  which is natural since the midpoint integrator being used is of second order.\\

We also consider another example.  Let us choose a Lagrangian of the  forced harmonic oscillator problem defined by $L(t,x,\dot x)=\dot x^2/2-x^2/2+x\, f(t)$  with a non-vanishing function $f$. So that  the equation \eqref{ContFracDamp:a}, again for $\alpha=\beta$, reads
\begin{equation}\label{eq:torvik}
    \ddot x+\mu\, D_-^{(2\alpha)} x+x=f(t),\quad t\in [0,1].
\end{equation}
For a non-integer $2\alpha$, this problem is known as Bagley-Torvik equation which can be used to describe, for example   the dynamics
of a rigid plate immersed in a Newtonian fluid when  $\alpha=3/4$ (see \cite{Ignor,Torvik}). Due to mathematical complexity, the analytic  solutions of  such  equation are very few and are restricted to the one  dimensional case. In particular, with the  initial conditions $x(0)=\dot x(0)=0$ and $\mu =1$, the Bagley-Torvik equation is exactly solvable  (see \cite{Raja,Ford}) by considering,
\begin{subequations}
\begin{align}
f(t) &= t^3  + 6 \, t + \frac{3.2}{\Gamma(1/2)}\,   t^2 \sqrt{t},\quad \alpha=\frac{1}{4} , \label{torvik1}\\
 f(t) &= \frac{15}{4} \sqrt{t} + \frac{15}{8} \sqrt{\pi} \, t   + t^2 \sqrt{t},\quad \alpha=  \frac{3}{4}.  \label{torvik2}
\end{align}
\end{subequations}
where the the analytic solutions are given, respectively, by $x(t)=t^3$ (resp.  $=t^{\frac{5}{2}}$). We
 solve numerically  the Bagley-Torvik problem \eqref{eq:torvik} using FVI on $[0,1]$ for $\alpha=  \frac{1}{4},\ \frac{3}{4}$. The global errors  (in logarithmic scale)  are presented in Figures \ref{fig:Torik}
 for  $h = 1/2^i,\ i=1,\ldots,8$.

\begin{figure}[H]
\centering
\includegraphics[width=.48\textwidth]{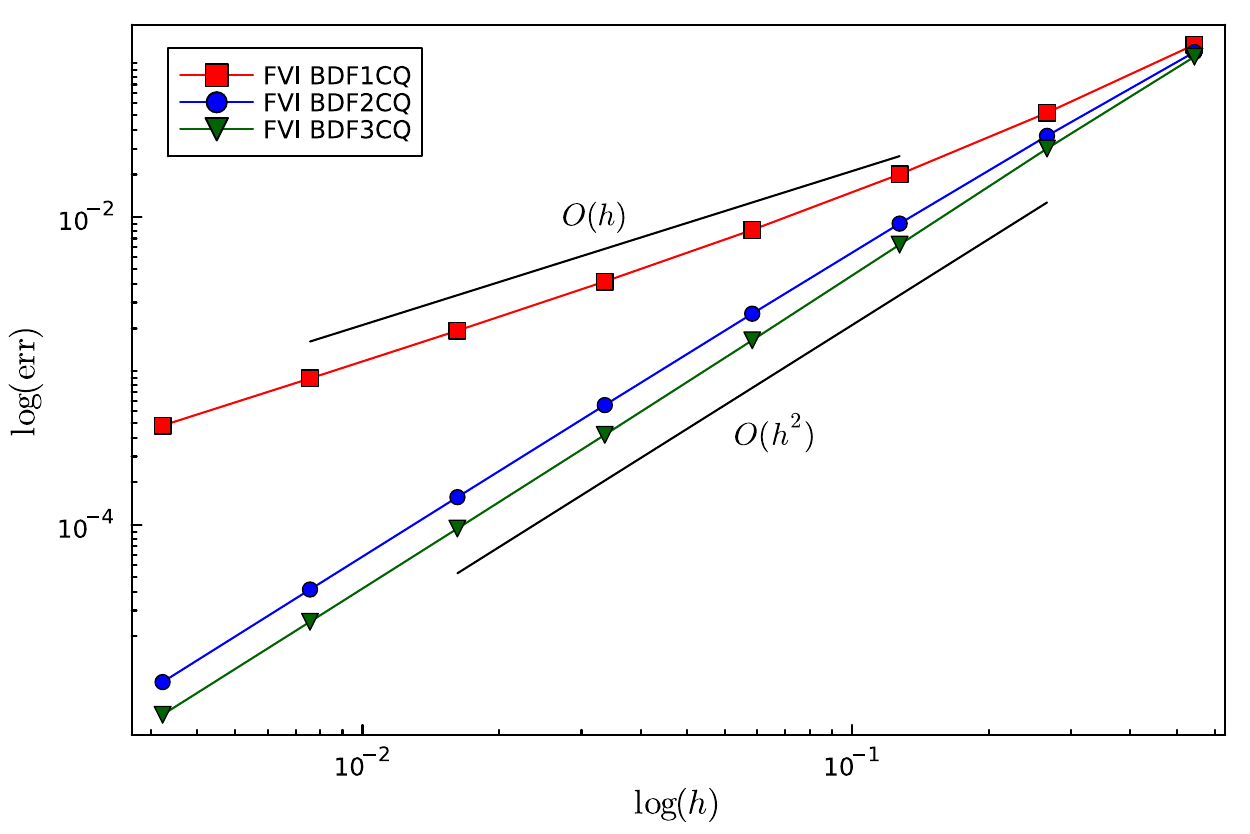}
\hfill
\includegraphics[width=.48\textwidth]{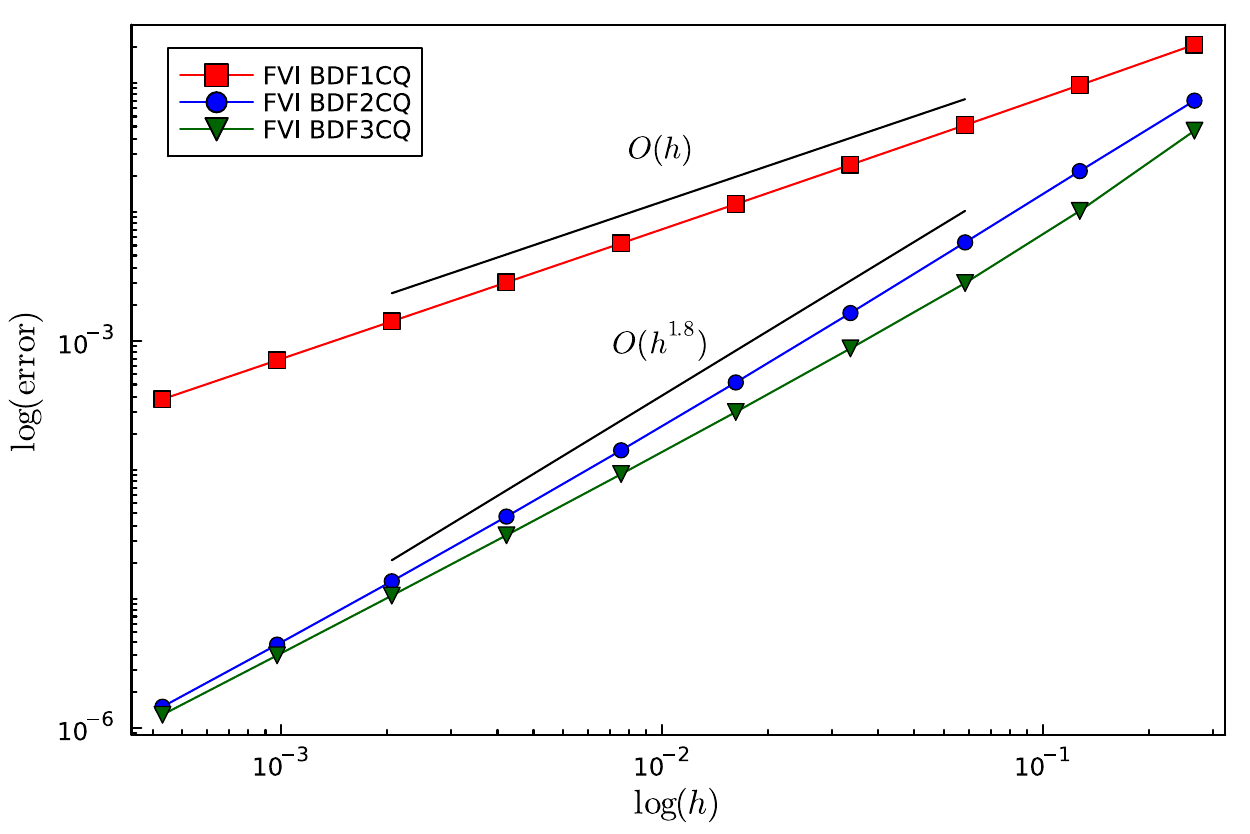}
\caption{Bagley-Torvik equation \eqref{eq:torvik}. Log-Log plot of the global errors on  $t\in [0,1]$ for  $h = 1/2^i,\ i=1,\ldots,8$. Left: case \eqref{torvik1}. Right: case \eqref{torvik2}.}
\label{fig:Torik}
\end{figure}

Again, it can be observed from Figure \ref{fig:Torik} that FVI-BDF1CQ  leading to a convergence of order one. A convergence of order $2$ for FVI-BDF2CQ and FVI-BDF3CQ is obtained (left plot) whereas a convergence of order cannot reach two for FVI-BDF2CQ and FVI-BDF3CQ (right plot).

We summarize the convergence order of  \eqref{DiscFracDamp:a}  for equations \eqref{eq:damped-oscillator}  and \eqref{eq:torvik}  in Table \ref{tab:1} where we consider the midpoint integrator for the conservative part.

\begin{table}[H]
    \centering
    {\renewcommand{\arraystretch}{1.5}%
\setlength{\tabcolsep}{12pt}
\begin{tabular}{|l|ccc|}
\hline
&BDF1CQ &BDF2CQ  & BDF3CQ\\  \hline
Damped oscillator ($\alpha=1/2$)& order 1 & order 2  & order 2\\
Baglay-Torvik ($\alpha=1/4$)& order 1 & order 2  & order 2\\
Baglay-Torvik ($\alpha=3/4$)& order 1 & order 1.8  & order 1.8\\
\hline
\end{tabular}}
    \caption{Convergence order of  \eqref{DiscFracDamp:a}  for equations \eqref{eq:damped-oscillator}  and \eqref{eq:torvik}.}
    \label{tab:1}
\end{table}

\section{Higher-order fractional variational integrators bosed on convolution quadrature}\label{FVI-CQ}

Now, we establish a particular discretization of the action \eqref{FracAction}, where we choose a higher-order approximation with quadrature rule $(b_i,c_i)_{i=1}^r$ (\S\ref{HO-Action}) for the conservative part $\Sc$ and convolution quadrature \eqref{DiscIntMinus}, \eqref{DiscIntPlus} for the fractional integrals involved in $\Sf$ instead of the order 1 \eqref{Order1CQ}. For that, we take into account two discrete series $x_d=\lc x_k\rc_{0:N}\in(\R^d)^{N+1}$, $y_d=\lc y_k\rc_{0:N}\in(\R^d)^{N+1}$ and $s+1$ inner nodes $\lc x_k^{\nu}\rc^{0:s}\in (\R^d)^{s+1}$ in each interval $[k,k+1]$ such that $x_k^s=x_{k+1}^0$ (equiv. for $y$). Namely
\begin{equation}\label{DiscFracActionHO}
\begin{split}
\Sg(x_d,y_d)=\Sc_d&(x_d,y_d)+\Sf_d(x_d,y_d), \\
\Sc_d(x_d,y_d)=\sum_{k=0}^{N-1}(L_d(x_k)+L_d(y_k)),&\quad \Sf_d(x_d,y_d)=-\mu\,h\,\,\sum_{k=0}^{N}\mathcal{J}^{-\alpha}_{-}x_k\,\mathcal{J}^{-\beta}_{+}y_k,\\
L_d(x_k)=h\sum_{i=1}^rb_iL(x_d(c_i\,h;k),\,\dot x_d(c_i\,h;k)),&\,\,\,\,\,\, L_d(y_k)=h\sum_{i=1}^rb_iL(y_d(c_i\,h;k),\dot y_d(c_i\,h;k)),
\end{split}
\end{equation}
where the definition \eqref{Polynomials} applies for $x_d(t;k)$ and $y_d(t;k)$ just by $Q=\R^d$. Now, considering restricted varied curves
\begin{equation}\label{RestrictedVariedInner}
x_d^{\epsilon}=\lc x_k^{\nu}\rc_{0:N-1}^{0:s}+\epsilon\, \lc\delta x_k^{\nu}\rc_{0:N-1}^{0:s},\,\,\,y_d^{\epsilon}=\lc y_k^{\nu}\rc_{0:N-1}^{0:s}+\epsilon\, \lc\delta x_k^{\nu}\rc_{0:N-1}^{0:s},
\end{equation}
such that $\delta x_0=\delta x_0^0=0$ and $\delta x_N=\delta x_{N-1}^s=0$, we establish the following result (it is important to recall that, from now on, we shall consider the variation operator as $\delta\equiv d/d\epsilon|_{\epsilon=0}$, applied over the ``varied'' quantities). Before the theorem, we set a useful lemma.

\begin{lemma}\label{VariationComm}
According to the definitions \eqref{DiscIntMinus}, \eqref{DiscIntPlus} and considering varied curves \eqref{RestrictedVariedInner}, we have that 
\[
\delta\, \mathcal{J}^{-\alpha}_{\lambda}x_k = \mathcal{J}^{-\alpha}_{\lambda}\,\delta x_k.
\]
Equivalently for $y$.
\end{lemma}
\begin{proof}
We pick $\sigma=-$, the proof for $+$ is equivalent. It is important to remark that in the convolution quadrature \eqref{DiscIntMinus} the inner nodes are not involved, and consequently from \eqref{RestrictedVariedInner} we only take into consideration the main nodes, i.e.~ $x_d^{\epsilon}=\lc x_k\rc_{0:N}+\epsilon\, \lc\delta x_k\rc_{0:N}$.
\[
\begin{split}
\delta\, \mathcal{J}^{-\alpha}_{-}x_k=\frac{d}{d\epsilon}\Big|_{\epsilon=0}\, \mathcal{J}^{-\alpha}_{-}x_d^{\epsilon}=\frac{d}{d\epsilon}\Big|_{\epsilon=0} \sum_{n=0}^{k}\omega_n^{(-\alpha)}(x_{k-n}+\epsilon\,\delta x_{k-n})=\sum_{n=0}^{k}\omega_n^{(-\alpha)}\delta x_{k-n}=\mathcal{J}^{-\alpha}_{-}\,\delta x_k.
\end{split}
\]
\end{proof}

\begin{theorem}\label{TheoremCQ}
The equations 
\begin{subequations}\label{DELHO}
\begin{align}
&D_{s+1}L_d(x_{k-1}^0, \ldots,x_{k-1}^s)+ D_1L_d(x_k^0, \ldots,x_k^s)-\mu\,h\,\mathcal{J}^{-(\alpha+\beta)}_{-}x_k^0=0,\qquad k=1,\ldots, N-1,\label{DELHOx:a}\\
&D_iL_d(x_k^0, \ldots,x_k^s)=0,\,\hspace{4.3cm}\,\,\quad k=0,\ldots, N-1,\,\quad i=2,\ldots,s,\label{DELHOx:b}\\
&D_{s+1}L_d(y_{k-1}^0, \ldots,y_{k-1}^s)+ D_1L_d(y_k^0, \ldots,y_k^s)-\mu\,h\,\mathcal{J}^{-(\alpha+\beta)}_{+}y_k^0=0,\qquad k=1,\ldots, N-1,\label{DELHOy:a}\\
&D_iL_d(y_k^0, \ldots,y_k^s)=0,\,\hspace{4.3cm}\,\,\quad k=0,\ldots, N-1,\qquad i=2,\ldots,s, \label{DELHOy:b}
\end{align}
\end{subequations}
are  sufficient conditions for the extremals of \eqref{DiscFracActionHO} under restricted calculus of variations \eqref{RestrictedVariedInner}.
\end{theorem}
 
\begin{proof}
From \eqref{DiscFracActionHO} we have that
\[
\delta \mathcal{L}_d(x_d,y_d)=\delta\,\Sc_d(x_d,y_d)+\delta\,\Sf_d(x_d,y_d).
\]
Let start to simplify $\delta\,\Sc_d(x_d,y_d)$
    \[
\begin{split}
\delta\,\Sc_d(x_d,y_d)=\sum_{k=0}^{N-1}\lp\frac{\der L_d(x_k)}{\der x_k^{\nu}}+\frac{L_d(y_k)}{\der y_k^{\nu}}\rp\,\delta x_k^{\nu},
\end{split}
\]
where the summation over $\nu$ is understood and we have employed the restricted variations \eqref{RestrictedVariedInner}. 

\medskip

Concerning the term $\delta\,\Sf_d(x_d,y_d)$ we have,  let us use the notation

\begin{equation}\label{Notat}
\frac{\der L_d(x_k)}{\der x_k^{\nu}}\delta x_k^{\nu}=D_iL_d(x_k)\,\delta x_k^{\nu_i},
\end{equation}
where on the left hand side $\nu=0, \ldots,s$ and on the right hand side $\nu_1=0,\,\nu_2=1,\, \ldots,\nu_{s+1}=s$ (in other words $D_i=\der/\der x_k^{\nu_i}$); this way, it is highlighted that $L_d$ is a function of $s+1$ variables (equiv. for $y$). Thus, we have  
\[
\begin{split}
\delta\,\Sc_d(x_d,y_d)=\sum_{k=0}^{N-1}\sum_{i=1}^{s+1}\lp D_iL_d(x_k)+D_iL_d(y_k)\rp\,\delta x_k^{\nu_i}.
\end{split}
\]

\medskip

\[
\begin{split}
\delta\,\Sf_d(x_d,y_d)=^1&-\mu\,h\,\,\sum_{k=0}^{N}\mathcal{J}^{-\alpha}_{-}x_k\,\mathcal{J}^{-\beta}_{+}\delta\, x_k -\mu\,h\,\,\sum_{k=0}^{N}\mathcal{J}^{-\alpha}_{-}\delta\, x_k\,\mathcal{J}^{-\beta}_{+}y_k\\
=^2&-\mu\,h\,\,\sum_{k=0}^{N}\mathcal{J}^{-\beta}_{-}\mathcal{J}^{-\alpha}_{-}x_k\,\delta\, x_k -\mu\,h\,\,\sum_{k=0}^{N}\delta\, x_k\,\mathcal{J}^{-\alpha}_{+}\mathcal{J}^{-\beta}_{+}y_k\\
=^3&-\mu\,h\,\,\sum_{k=1}^{N-1}\lp \mathcal{J}^{-(\beta+\alpha)}_{-}x_k^0+\mathcal{J}^{-(\alpha+\beta)}_{+}y_k^0\rp\,\delta\, x_k^0. 
\end{split}
\]
In $=^1$ we have employed the Leibnitz rule of the derivative and Lemma \ref{VariationComm}. In $=^2$ we have employed the asymmetric integration by parts, i.e.~ property (2) in Lemma \ref{ConvProperties}. Finally, in $=^3$ we have rearranged terms, employed the semigroup property (1) in Lemma \ref{ConvProperties}, taken into account that $x_k=x_k^0$, $y_k=y_k^0$ and $\delta\, x_k=\delta\,x_k^0$ in terms of the inner nodes and taken also into account that  $\delta x_0=\delta x_N=0$, such that the terms $k=0$ and $k=N$ vanish.

\medskip

Putting everything together we have

\[
\begin{split}
\delta\, \mathcal{L}_d(x_d,y_d)=&\sum_{k=0}^{N-1}\sum_{i=1}^{s+1}\lp D_iL_d(x_k)+D_iL_d(y_k)\rp\,\delta x_k^{\nu_i}-\mu\,h\,\,\sum_{k=1}^{N-1}\lp \mathcal{J}^{-(\alpha+\beta)}_{-}x_k^0+\mathcal{J}^{-(\alpha+\beta)}_{+}y_k^0\rp\,\delta\, x_k^0\\
=&\sum_{k=0}^{N-1}\sum_{i=2}^{s+1} D_iL_d(x_k)\,\delta x_k^{\nu_i}+\sum_{k=1}^{N-1}\lp D_1L_d(x_k)-\mu\,h\,\mathcal{J}^{-(\alpha+\beta)}_{-}x_k^0\rp\delta\, x_k^0\\
+&\sum_{k=0}^{N-1}\sum_{i=2}^{s+1} D_iL_d(y_k)\,\delta x_k^{\nu_i}+\sum_{k=1}^{N-1}\lp D_1L_d(y_k)-\mu\,h\,\mathcal{J}^{-(\alpha+\beta)}_{+}y_k^0\rp\delta\, x_k^0\\
\end{split}
\]
where is taken into account that $\delta x_0=\delta x_0^0=0$ and that $D_{s+1}L_d(x_{k-1})=D_1L_d(x_k).$ Now, given that $\delta x_k^{\nu_i}$ are arbitrary for $k=0,\ldots,N-1$, $i=1,\ldots,s+1$ (except $\delta x_0$), we see from the last equality that
\[
\begin{split}
&D_iL_d(x_k)=0,\,\hspace{5cm}\,\,\quad k=0,\ldots, N-1,\,\quad i=2,\ldots,s,\\
&D_{s+1}L_d(x_{k-1})+ D_1L_d(x_k)-\mu\,h\,\mathcal{J}^{-(\alpha+\beta)}_{-}x_k=0,\,\,\,\,\,\, k=1,\ldots, N-1,\\
&D_iL_d(y_k)=0,\,\hspace{5cm}\,\,\quad k=0,\ldots, N-1,\,\quad i=2,\ldots,s,\\
&D_{s+1}L_d(y_{k-1})+ D_1L_d(y_k)-\mu\,h\,\mathcal{J}^{-(\alpha+\beta)}_{+}y_k=0,\,\,\,\,\,\,\, k=1,\ldots, N-1,
\end{split}
\]
is a sufficient condition for $\delta\, \mathcal{L}_d(x_d,y_d)=0$ and the claim holds.
\end{proof}

\begin{remark}
The partial derivatives of $L_d(x_k)$ (equiv. $y$) are completely determined by the quadrature rule $(b_i,c_i)_{i=1}^r$ and the Lagrangian function $L(q,\dot q)$, according to \eqref{DiscFracActionHO} and \eqref{Polynomials}. Namely
\[
\begin{split}
\frac{\der L_d(x_k)}{\der x_k^{\nu}}=&h\sum_{i=1}^rb_i\lp\frac{\der L}{\der q}(x_d(c_i\,h;k),\,\dot x_d(c_i\,h;k))\frac{\der x_d}{\der x_k^{\nu}}+\frac{\der L}{\der\dot q}(x_d(c_i\,h;k),\,\dot x_d(c_i\,h;k))\frac{\der \dot x_d}{\der x_k^{\nu}}\rp\\
=&h\sum_{i=1}^rb_i\lp\frac{\der L}{\der q}(x_d(c_i\,h;k),\,\dot x_d(c_i\,h;k))\,\ell_{\nu}(c_ih)+\frac{\der L}{\der\dot q}(x_d(c_i\,h;k),\,\dot x_d(c_i\,h;k))\,\frac{1}{h}\dot \ell_{\nu}(c_ih)\rp.
\end{split}
\]
\end{remark}

Naturally, equations \eqref{DELHOx:a},\eqref{DELHOx:b} can be employed as a discrete iteration scheme for the dynamics \eqref{ContFracDamp:a}, the same way \eqref{DELHOy:a},\eqref{DELHOy:b} can be used for \eqref{ContFracDamp:b}. We shall focus on the $x$-part, since $y$ is equivalent. 

The equations
\[
\begin{split}
p_{x_0}:=&-D_1L_d(x_0^0, \ldots,x_0^s),\\
0=&\,\,\,\,\,\,\,D_iL_d(x_k^0, \ldots,x_k^s),\quad \forall\,i=2,\ldots,s,\,\hspace{4cm}\,\,\,\quad k=0,\ldots, N-1,\\
0=&\,\,\,\,\,\,\,D_{s+1}L_d(x_{k-1}^0, \ldots,x_{k-1}^s)+ D_1L_d(x_k^0, \ldots,x_k^s)-\mu\,h\,\mathcal{J}^{-(\alpha+\beta)}_{-}x_k^0,\,\,\,\,\,\,\,\,\,\,\,\, k=1,\ldots, N-1,
\end{split}
\]
(where we include the initial momentum $p_{x_0}:=-D_1L_d(x_0^0, \ldots,x_0^s)$ as a definition \footnote{Naturally, this definition is based on the Hamiltonian version of discrete mechanics, which can be consulted for conservative systems in \cite{MaWe}, and for the particular case of fractional damping in \cite{JiOb2}. We do not enter here in further details since it is offtopic.})
conform a discrete iteration scheme 
\[
\begin{split}
x_0^0&\,\,\mapsto\,\, (x_0^0, \ldots,x_0^s=x_1^0),\\
(x_{k-1}^0, \ldots,x_{k-1}^s=x_k^0)&\,\,\mapsto\,\,(x_{k}^0, \ldots,x_{k}^s=x_{k+1}^0),\quad k=1, \ldots,N-1,
\end{split}
\]
that can be represented  as an algorithm:

\begin{algorithm}{}
\begin{algorithmic}[1]
\State {\bf Initial data}: $N,\, h,\,\alpha,\,\beta\,,\,\omega_n^{(-(\alpha+\beta))},\,\mu,\, x_0^0,\, p_{x_0}.$
 \State {\bf solve for} $x_0^1, \ldots,x_0^s$ {\bf from} 

\[
\begin{split}
p_{x_0}=&-D_1L_d(x_0^0, \ldots,x_0^s),\\
0=&\,\,\,\,\,\,\,D_iL_d(x_0^0, \ldots.,x_0^s),\,\,\quad \forall\,i=2, \ldots,s.
\end{split}
\]
\State {\bf Initial points:} $x_0^0, \ldots,x_0^s=x_1^0$
    \For {$k= 1: N-1$} 
    
\hspace{-0.6cm} {\bf solve for} $x_{k}^1, \ldots,x_k^s=x_{k+1}^0$ {\bf from} 
\[
\begin{split}
0=&\,\,D_{s+1}L_d(x_{k-1}^0, \ldots,x_{k-1}^s)+ D_1L_d(x_k^0, \ldots,x_k^s)-\mu\,h\,\sum_{n=0}^k\omega_n^{(-(\alpha+\beta))}x_{k-n}^0,\\
0=&\,\,D_iL_d(x_k^0, \ldots,x_k^s),\quad\quad \forall\,i=2, \ldots,s.
\end{split}
\]
    \EndFor
    \State  {\bf Output:} $(x_1^{\nu}, \ldots,x_{N-1}^{\nu}),\,\,\,\nu=0, \ldots,s.$
\end{algorithmic}
\caption{Higher-order fractional variational integrator   (with convolution quadrature)}\label{alg:FractionalAlgorithm}

  \end{algorithm}

It is important to remark that at each step a nonlinear system of $s$ algebraic equations is solved in order to obtain the $s$ unknowns $(x_k^1, \ldots,x_k^s=x_{k+1}^0)$, even in the initialization step.

\subsection{Numerical experiment}  We will employ a variational integrators of order $4$ for the conservative part based on two points Gauss quadrature  and a  polynomial degree $2$ (see \cite{SinaSaake}) and BDFCQ for the fractional one. To simplify notation, we continue to write FVI-BDFCQ (or only FVI  when it is convenient) for equations \ref{DELHOx:a} and \ref{DELHOx:b} and  the  numerical solution can be computed using  Algorithm \ref{alg:FractionalAlgorithm}.\\

Let us consider the previous  examples as in \S \ref{subsec:experiment}.  
As expected, we notice in Figures \ref{fig:FVIcovergence-integer1} and \ref{fig:Torik2} that FVI-BDF1CQ and  FVI-BDF2CQ are of first and second order, respectively. From the
numerical point of view one would expect a higher accuracy (order 3 using BDF3CQ mixed with third order variational integrator) which we do not get. One possible reason might be mixing of integrators (VI and BDFCQ). In particular,  BDFCQ depends only the  main nodes but not on the inner nodes which are  considered in the conservative part. Another possible reason might be saturation effects coming from CQ as described in \S \ref{ConQua}.

\begin{figure}[H]
\centering
\includegraphics[width=.48\textwidth]{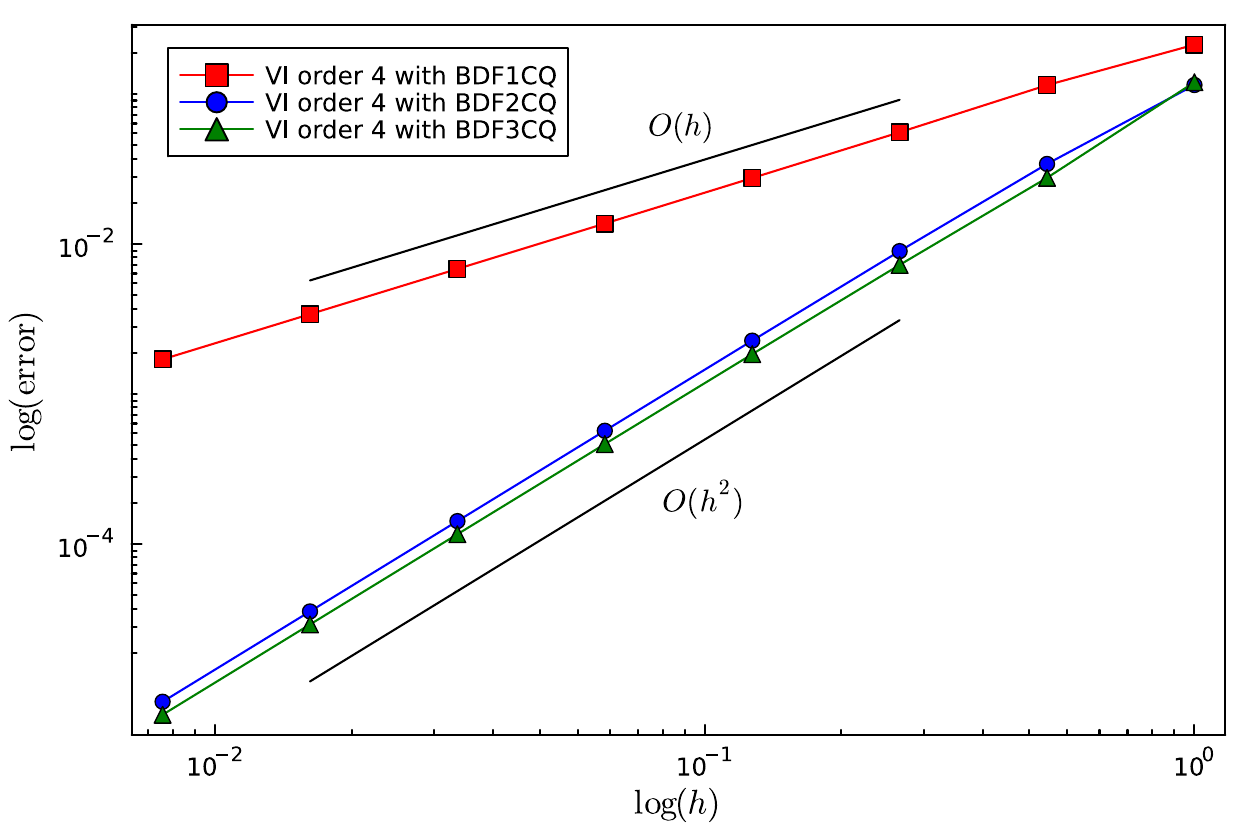}
\caption{Damped harmonic oscillator \eqref{eq:damped-oscillator}.  Log-Log plot of the global error presented on $t\in [0,16]$ for $h =16/2^i,\ i=4,\ldots,11$.}
\label{fig:FVIcovergence-integer1}
\end{figure}

\begin{figure}[H]
\centering
\includegraphics[width=.48\textwidth]{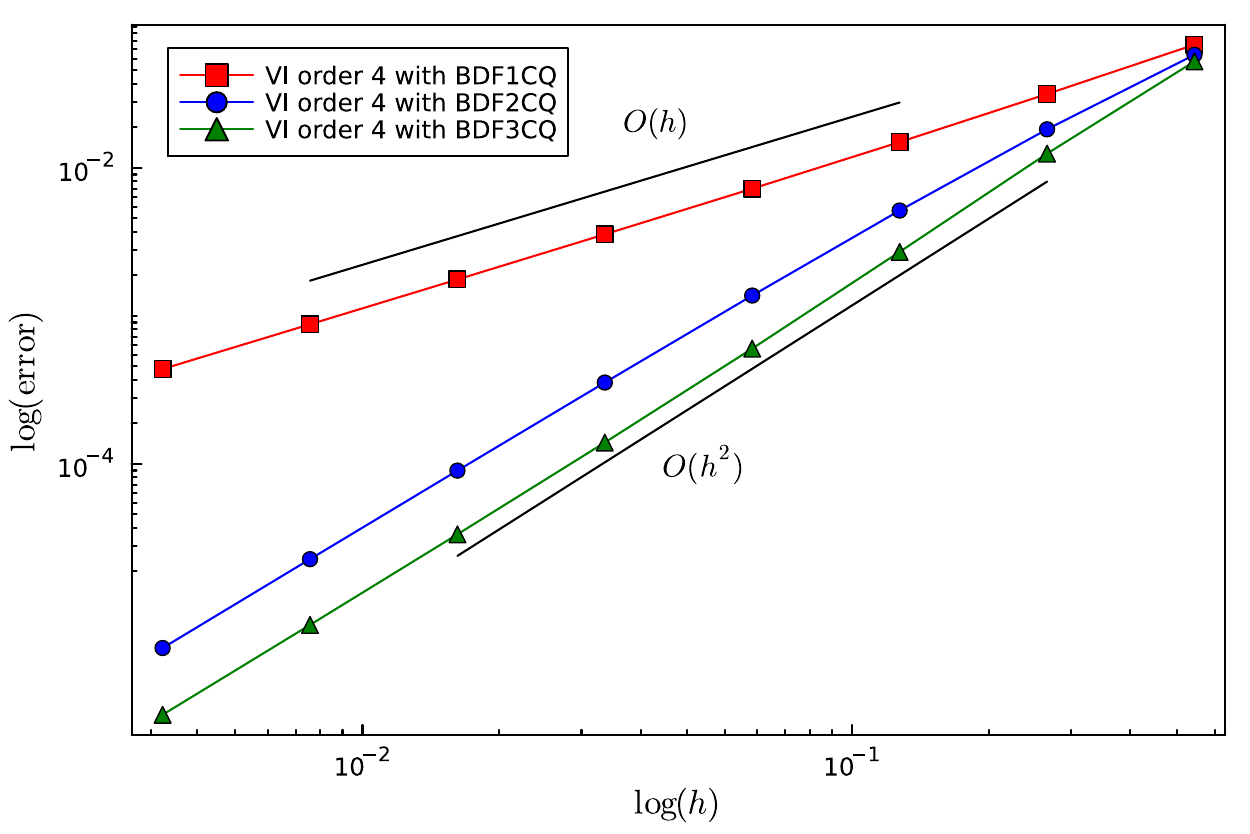}
\hfill
\includegraphics[width=.48\textwidth]{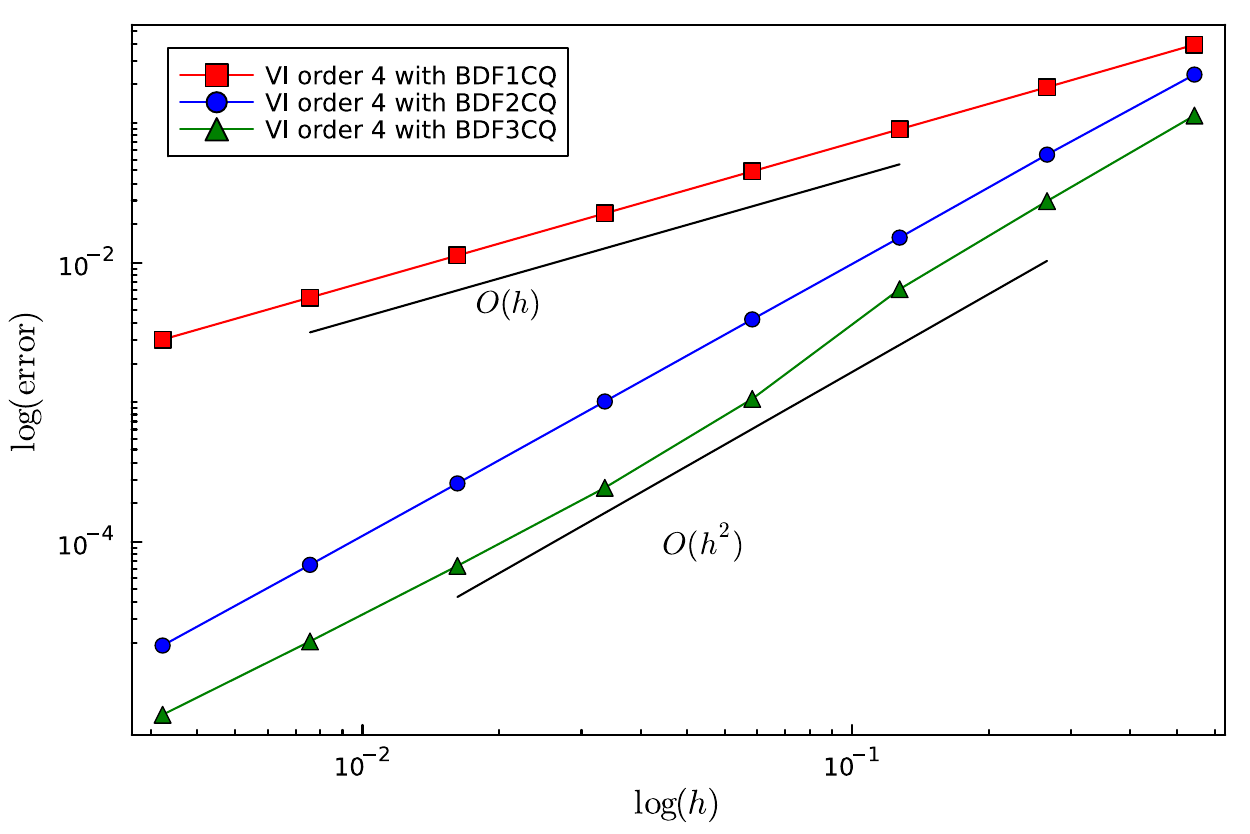}
\caption{Bagley-Torvik equation \eqref{eq:torvik}. Log-Log plot of the global error on  $t\in [0,1]$ for  $h = 1/2^i,\ i=1,\ldots,8$. Left: case \eqref{torvik1}. Right: case \eqref{torvik2}.}
\label{fig:Torik2}
\end{figure}

As we have seen in \S\ref{ConQua}, the main issue of  using BDFCQ for certain class of solution functions is that  one cannot achieve a high accuracy, see the saturation effects in Figures \ref{Saturation}. However, a correction term should be added as in \eqref{StartingQuad} to recover the order of accuracy as the one of the underlying BDF methods. We apply BDF3CQ with  a correction term in Algorithm \ref{alg:FractionalAlgorithm} for equation \eqref{eq:torvik} when $\alpha=3/4$ and as  we observe in Figure \ref{fig:Torik3},  the third order accuracy is almost achieved. However, this phenomenon does not work with the previous studied cases which  seems related to the accumulation of  errors. 

\begin{figure}[H]
\centering
\includegraphics[width=.48\textwidth]{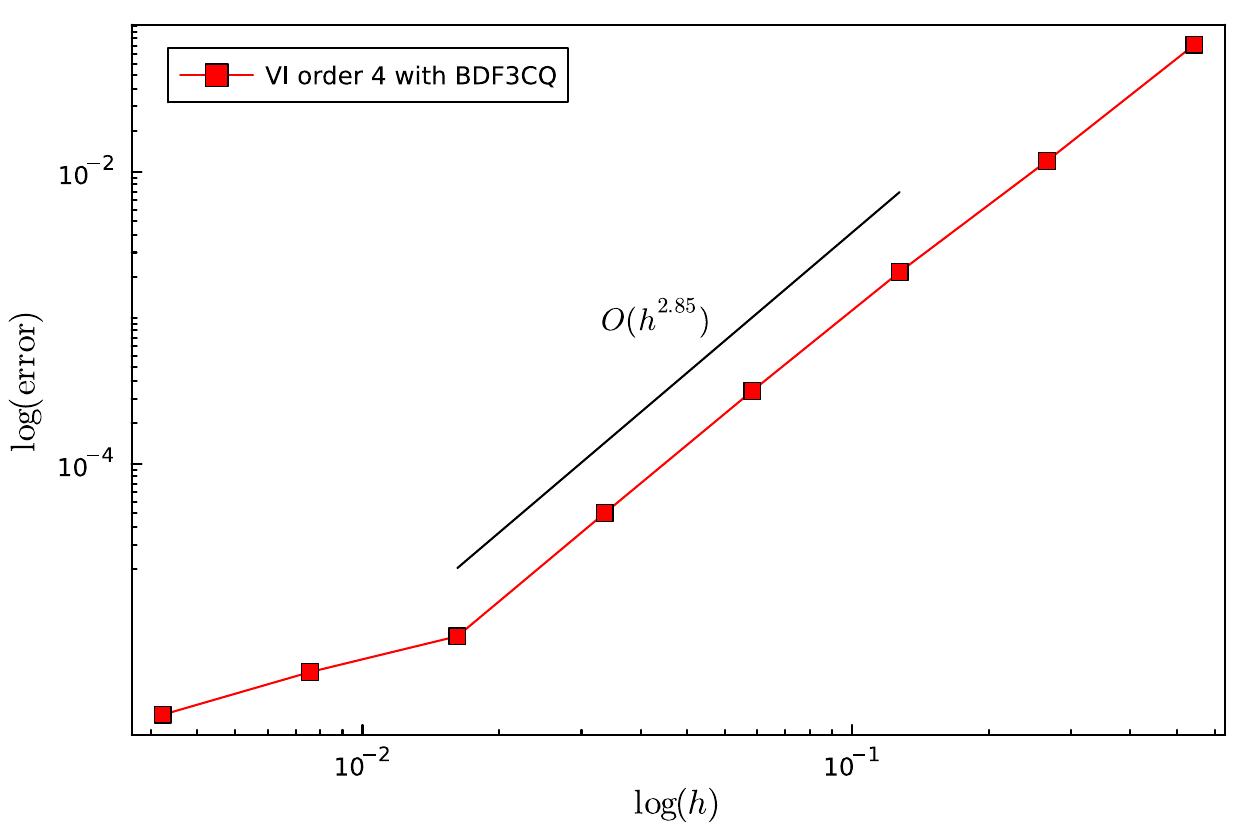}
\caption{Bagley-Torvik equation \eqref{eq:torvik},  case \eqref{torvik2}, i.e.~ $\alpha=3/4$. Log-Log plot of the global error on  $t\in [0,1]$ for  $h = 1/2^i,\ i=1,\ldots,8$.}
\label{fig:Torik3}
\end{figure}


We summarize the convergence order of  \eqref{DiscFracDamp:a}  for equations \eqref{eq:damped-oscillator}  and \eqref{eq:torvik}  in Table \ref{tab:2} where we consider an integrator of order $4$ for the conservative part. 

\begin{table}[H]
    \centering
    {\renewcommand{\arraystretch}{1.5}%
\setlength{\tabcolsep}{12pt}
\begin{tabular}{|l|ccc|}
\hline
&BDF1CQ &BDF2CQ  & BDF3CQ\\  \hline
Damped oscillator ($\alpha=1/2$)& order 1 & order 2  & order 2\\
Baglay-Torvik ($\alpha=1/4$)& order 1 & order 2  & order 2\\
Baglay-Torvik ($\alpha=3/4$)& order 1 & order 2  & order 2\\
\hline
\end{tabular}}
    \caption{Convergence order of  \eqref{DiscFracDamp:a}  for equations \eqref{eq:damped-oscillator}  and \eqref{eq:torvik}.}
    \label{tab:2}
\end{table}

\section{Conclusions}
A restricted Hamilton's principle is a new class of fractional calculus of variations  has been introduced  in \cite{JiOb1,JiOb2}.  
The main motivation of this approach is to derive the dynamics of fractionally damped systems \eqref{eq:FV} using a purely variational way and hence  to construct the so-called fractional variational integrators (FVIs). \\

We have developed FVIs  that combine  the convolution quadrature (CQ), which is particularly suitable  \cite{Lubich1,Lubich2} in the framework of the restricted Hamilton's principle, with the variational integrators \cite{MaWe,SinaSaake,HaLe13}. Our result coincides, in particular, with the one given in \cite{JiOb2} when using the classical variational integrators.\\

This work centers around increasing the accuracy the numerical scheme associated to  \eqref{eq:FV}. Here,  we have focused on implementing the FVIs  and  test numerically  their accuracy using two mechanical systems, the damped harmonic oscillator and the Bagley-Torvik problems. We notice that for FVI based on BDFCQ, it can only achieve the second-order accuracy even for a higher-order FVI (see Figures \ref{fig:FVIcovergence-integer1} and \ref{fig:Torik2}) which is due to the fact that  saturation effects are also a part of the problem. In this situation,  with the use of correction term, the third-order accuracy for FVI-BDF3CQ is  observed in Figure \eqref{fig:Torik3}.\\

To overcome the problem of limitation with this strategy i.e.~  in order to obtain the same order of convergence of the underlying BDF methods,   it will be necessary to take a correction term  in to account which is  difficult, in general,  to deal with for some values of $\alpha$ and  further  errors  are arising  from  solving the linear systems of  the starting quadrature  \cite{Dlhm}.\\

Another problem of using  BDFCQ is that, the inner nodes used in a higher-order approximation for the conservative action \eqref{FracAction} are not taken into account in  BDFCQ for the fractional one. Thus, a way to handle this is to apply the high-order Runge–Kutta convolution quadrature (RKCQ) \cite{LuOs,LuBa} for the fractional part which will be a future work.

\newpage

\printbibliography

@article{Agra,
author = {Om P. Agrawal},
title = {Formulation of {E}uler–{L}agrange equations for fractional variational problems},
journal = {Journal of Mathematical Analysis and Applications},
volume = {272},
number = {1},
pages = {368-379},
year = {2002}
}

@article{Bateman,
  title = {On Dissipative Systems and Related Variational Principles},
  author = {Bateman, H},
  journal = {Phys. Rev.},
  volume = {38},
  pages = {815--819},
  year = {1931},
}

@Inbook{Ca13,
author="Campos, C{\'e}dric M",
title="High Order Variational Integrators: A Polynomial Approach",
bookTitle="Advances in Differential Equations and Applications",
year="2014",
publisher="Springer International Publishing",
pages="249--258",
editor="Casas, Fernando and Mart{\'i}nez, Vicente",
volume    = "4",
}

@article{Ca14,
author = {Cédric M. Campos and Sina Ober-Blöbaum and Emmanuel Trélat},
title = {High order variational integrators in the optimal control of mechanical systems},
journal = {Discrete and Continuous Dynamical Systems},
volume = {35},
number = {9},
pages = {4193--4223},
year = {2015}
}

@article{Cresson1,
author = {Loïc Bourdin and Jacky Cresson and Isabelle Greff and Pierre Inizan},
title = {Variational integrator for fractional {E}uler–{L}agrange equations},
journal = {Applied Numerical Mathematics},
volume = {71},
pages = {14--23},
year = {2013},
}

@article{Cresson,
author = {Jacky Cresson and Pierre Inizan},
title={Variational formulations
of differential equations and asymmetric fractional embedding},
journal = {Journal of Mathematical Analysis and Applications},
volume = {385},
number = {2},
pages = {975--997},
year = {2012},
}

@article{Galley,
  title = {Classical Mechanics of Nonconservative Systems},
  author = {Galley, Chad R.},
  journal = {Phys. Rev. Lett.},
  volume = {110},
  pages = {174301},
  year = {2013},
}

@book{TheBook,
   title =     {Fractional Integrals and Derivatives: Theory and Applications},
   author =    {Stefan G. Samko and Anatoly A. Kilbas and Oleg I. Marichev},
   publisher = {Gordon},
   year =      {1993},
   edition =   {1},
}

@book{TheBook2,
   title =     {Theory and Applications of Fractional Differential Equations},
   author =    {Anatoly A. Kilbas, Hari M. Srivastava and Juan J. Trujillo },
   publisher = {Elsevier},
   year =      {2006},
   series =    {North-Holland Mathematics Studies 204},
   edition =   {1},
   volume =    {},
}

@book{Oldham,
   title =     {The fractional calculus: Theory and Applications of Differentiation
and Integration to Arbitrary Order},
   author =    {Keith B. Oldham and Jerome Spanier},
   publisher = {Academic Press, New York, London},
   year =      {1974},
}

@article{Riewe,
  title = {Nonconservative Lagrangian and Hamiltonian mechanics},
  author = {Riewe, Fred},
  journal = {Phys. Rev. E},
  volume = {53},
  number= {2},
  pages = {1890--1899},
  year = {1996},
  publisher = {American Physical Society},
}

@article{SinaSaake,
  title = {Construction and analysis of higher order Galerkin variational integrators},
  author = {Ober-Blöbaum, Sina and Saake, Nils},
  journal = {Advances in Computational Mathematics},
  volume = {41},
  pages = {955--986},
  year = {2015}
}

@article{O14,
    author = {Ober-Blöbaum, Sina},
    title = "{Galerkin variational integrators and modified symplectic Runge–Kutta methods}",
    journal = {IMA Journal of Numerical Analysis},
    volume = {37},
    number = {1},
    pages = {375--406},
    year = {2016},
}

@article{MoVe,
    author = {Moser, Jürgen and Veselov, Alexander P},
    title = "{Galerkin variational integrators and modified symplectic Runge–Kutta methods}",
    journal = {IMA Journal of Numerical Analysis},
    volume = {139},
    pages = {217--243 },
    year = {1991},
}

@article{MaWe, 
author={Marsden, J. E. and West, M.},
title={Discrete mechanics and variational integrators},
volume={10},
journal={Acta Numerica},
publisher={Cambridge University Press},
year={2001}, 
pages={357--514}}

@ARTICLE{LuOs,
       author = {Lubich, C. and Ostermann, A.},
        title = "{Runge-Kutta methods for parabolic equations and convolution quadrature}",
      journal = {Mathematics of Computation},
         year = 1993,
       volume = {60},
       number = {201},
        pages = {105--131},
}

@ARTICLE{LuBa,
       author = {Banjai, Lehel and Lubich, Christian},
        title = "{RAn error analysis of Runge–Kutta convolution quadrature}",
      journal = {BIT Numerical Mathematics},
         year = 2011,
       volume = {51},
       number = {3},
        pages = {483--496},
}

@article{Lubich1,
author = {Lubich, C.},
title = {Discretized Fractional Calculus},
journal = {SIAM Journal on Mathematical Analysis},
volume = {17},
number = {3},
pages = {704--719},
year = {1986}
}

@ARTICLE{Lubich2,
       author = {Lubich, C},
        title = "{Convolution quadrature and discretized operational calculus. I and II}",
      journal = {Numerische Mathematik},
         year = {1988},
       volume = {52},
        pages = { 129--145 and  413--425},
}

@ARTICLE{Lubich3,
       author = {Lubich, C},
        title = {On the multistep time discretization of linear initial-boundary value problems and their boundary integral equations},
      journal = {Numerische Mathematik},
         year = {1994},
       volume = {67},
        pages = {365--389},
}

@ARTICLE{LubichF1,
       author = {Lubich, C},
        title = "{Fractional linear multistep methods for Abel–Volterra integral equations of
the second kind}",
      journal = {Math. Comp.},
         year = {1985},
       volume = {45},
        pages = { 463--469},
}

@ARTICLE{LubichF2,
       author = {Lubich, C},
        title = "{A stability analysis of convolution quadratures for Abel-Volterra integral
equations}",
      journal = {IMA J. Numer. Anal.},
         year = {1986},
       volume = {6},
        pages = { 87--101},
}

@article{SinaMats,
title = {Superconvergence of Galerkin variational integrators},
journal = {IFAC-PapersOnLine},
volume = {54},
number = {19},
pages = {pp .327--333},
year = {2021},
author = {Sina Ober-Blöbaum and Mats Vermeeren},
}

@article{Lavoie,
author = {Lovoie, J. L. and Osler, T. J. and Tremblay, R.},
title = {Fractional Derivatives and Special Functions},
journal = {SIAM Review},
volume = {18},
number = {2},
pages = {240--268},
year = {1976},
}

@article{Leok2011,
author = {Leok, Melvin and Shingel, Tatiana},
title = {General techniques for constructing variational integrators},
journal = {Frontiers of Mathematics in China},
volume = {7},
number = {2},
pages = {273--303},
year = {2012},
}

@book{Kai,
   author =    {Kai Diethelm},
   title =     {The Analysis of Fractional Differential Equations: An Application-Oriented Exposition Using Differential Operators of Caputo Type},
   publisher = {Springer-Verlag Berlin Heidelberg},
   year =      {2010},
   series =    {Lecture Notes in Mathematics 2004},
   edition =   {1},
}

@article{JiOb1,
title = {A Fractional Variational Approach for Modelling Dissipative Mechanical Systems:Continuous and Discrete Settings},
author = {Fernando Jiménez and Sina Ober-Blöbaum},
journal = {IFAC-PapersOnLine},
volume = {51},
number = {3},
pages = {50-55},
year = {2018}
}

@article{JiOb2,
author = {Jiménez, Fernando  and Ober-Blöbaum, Sina},
title = {Fractional Damping Through Restricted Calculus of Variations},
journal = {Journal of Nonlinear Science},
volume = {31},
pages = {46},
year = {2021}
}

@article{HaLe13,
author = {Hall, James  and Leok, Melvin},
title = {Spectral variational integrators},
journal = {Numerische Mathematik},
volume = {130},
pages = {681--740},
number ={4},
year = {2015}
}

@book {LuGeWa,
    AUTHOR = {Hairer, Ernst and Lubich, Christian and Wanner, Gerhard},
     TITLE = {Geometric numerical integration},
    SERIES = {Springer Series in Computational Mathematics},
    VOLUME = {31},
   EDITION = {Second},
      NOTE = {Structure-preserving algorithms for ordinary differential
              equations},
 PUBLISHER = {Springer-Verlag, Berlin},
      YEAR = {2006},
     PAGES = {xviii+644},
      ISBN = {3-540-30663-3; 978-3-540-30663-4},
}

@book{AbMa,
   author =    {Abraham, R. and  Marsden, J. E.},
   title =     {Foundations of Mechanics},
   publisher = {Benjamin/Cummings Publishing Company},
   year =      {1978}
}

@article{Torvik,
    author = {Torvik, P. J. and Bagley, R. L.},
    title = {On the Appearance of the Fractional Derivative in the Behavior of Real Materials},
    journal = {Journal of Applied Mechanics},
    volume = {51},
    number = {2},
    pages = {294--298},
    year = {1984},
}

@article{Raja,
    author = {Jena, Rajarama Mohan and Chakraverty, S.},
    title = "{Analytical solution of Bagley-Torvik equations using Sumudu transformation method}",
    journal = {SN Applied Sciences},
    volume = {1},
    pages = {246},
    year = {2019},
}

@article{Ford,
author = {Neville J. Ford and Joseph A. Connolly},
title = {Systems-based decomposition schemes for the approximate solution of multi-term fractional differential equations},
journal = {Journal of Computational and Applied Mathematics},
volume = {229},
number = {2},
pages = {382-391},
year = {2009},
}

@book{Ignor,
   title =     {Fractional Differential Equations: An Introduction to Fractional Derivatives, Fractional Differential Equations, to Methods of Their Solution and Some of Their Applications},
   author =    {Ignor Podlubny},
   publisher = {Academic Press},
   year =      {1998},
   series =    {Mathematics in Science and Engineering},
   edition =   {1st},
   number =    {198},
}

@article{Diego,
year = {2018},
publisher = {IOP Publishing},
volume = {31},
number = {8},
pages = {3814},
author = {D Martín de Diego and R Sato Martín de Almagro},
title = {Variational order for forced Lagrangian systems},
journal = {Nonlinearity},
}

@article{Dlhm,
title = {Pitfalls in fast numerical solvers for fractional differential equations},
journal = {Journal of Computational and Applied Mathematics},
volume = {186},
number = {2},
pages = {482-503},
year = {2006},
author = {Kai Diethelm and Judith M. Ford and Neville J. Ford and Marc Weilbeer},
}

\end{document}